\documentclass[11pt]{amsart}
\usepackage[english]{babel}

\numberwithin{equation}{section}

\date{\today}

\author{Zhuchao Ji and Weixiao Shen}

\title{The wandering domain problem for attracting polynomial skew products}

\address{Institute for Theoretical Sciences, Westlake University, Hangzhou,  China 310030}

\email{jizhuchao@westlake.edu.cn}

\address{Shanghai Center for Mathematical Sciences, Fudan University, Shanghai, China 200438}

\email{wxshen@fudan.edu.cn}


\newtheorem{theorem}{Theorem}[section]
\newtheorem{definition}[theorem]{Definition}
\newtheorem{proposition}[theorem]{Proposition}

\newtheorem{lemma}[theorem]{Lemma}
\newtheorem{remark}[theorem]{Remark}

\newtheorem{problem}[theorem]{Problem}

\newtheorem*{theorem*}{Theorem}
\newtheorem*{definition*}{Definition}
\newtheorem*{lemma*}{Lemma}
\newtheorem*{proposition*}{Proposition}
\newtheorem*{question*}{Question}

\begin{document}

\setcounter{tocdepth}{1}
\maketitle
\pagestyle{plain}
	
	\begin{abstract}
Wandering Fatou components were recently constructed by Astorg et al \cite{astorg2016two} for higher dimensional holomorphic maps on projective spaces. Their examples are polynomial skew products with a parabolic invariant line. In this paper we study this wandering domain problem  for polynomial skew product $f$ with an attracting invariant line $L$ (which is the more common case). We show that if $f$ is unicritical (in the sense that the critical curve has a unique transversal intersection with $L$), then every Fatou component of $f$ {\bf in the basin of $L$}
is an extension of a one-dimensional Fatou component of $f|_L$. As a corollary there is no wandering Fatou component.  We will also discuss the  multicritical case under additional assumptions.
	\end{abstract}
	\tableofcontents
\bigskip	
\section{Introduction}
\subsection{Backgrounds}
In one-dimensional complex dynamics, a celebrated theorem of Sullivan \cite{sullivan1985quasiconformal} asserts that there is no wandering Fatou component for rational maps on the Riemann sphere $\mathbb{P}^1$. This leads to a complete classification of Fatou components for rational maps: a Fatou component is preperiodic to an attracting basin, a parabolic basin, or a rotation domain.
\medskip
\par In higher dimension, let $f$ be a {\em holomorphic endomorphisms} on the projective space $\mathbb{P}^k$, $k\geq 2$. The Fatou set is classically defined as the maximal locus such that $\left\{f^n\right\}_{n\geq 1}$ form a normal family.  However, contrary to the one-dimensional case, wandering Fatou components were recently constructed by Astorg-Buff-Dujardin-Peters-Raissy \cite{astorg2016two} for  holomorphic endomorphisms on the projective plane $\mathbb{P}^2$.
\medskip
\par The counter examples constructed in \cite{astorg2016two} are  {\em polynomial skew products}. By definition a polynomial skew product is a self map on $\mathbb{C}^2$ of the following form:
\begin{equation*}
f(z,w)=(p(z),q(z,w)),
\end{equation*}
where $p,q$ are polynomials of degree at least 2. A polynomial skew product $f$ is called {\em regular} if it can be extended to a holomorphic endomorphism on $\mathbb{P}^2$, which is equivalent to the condition that $\text{deg}(p)=\text{deg}(q)=d$,  and $q(z,w)= c w^d+\;\text{lower degree terms}$, for some $c\neq 0$.  For a regular polynomial skew product, the line at infinity is an attracting set, and the Fatou components in the attracting basin of the line at infinity was classified by Lilov \cite{lilov2004fatou}. (In this case the line at infinity is {\em super-attracting} by the terminology we defined below.)
\medskip
\par Now assume that $\Omega$ is a  Fatou component of $f$ with bounded orbit, let $\pi$  be the projection map to the $z-$coordinate, then $\pi(\Omega)$ is contained in the Fatou set of $p$. Recall that the polynomial map $p$ does not have Herman rings. Since the classification of Fatou components of $p$ is known (by Sullivan's theorem), to investigate the dynamics of $f$ on $\Omega$, by passing to an iteration of $f$ and a coordinate change, we may assume that the line $L=\left\{z=0\right\}$ is invariant (i.e. $p(0)=0$), and $\Omega$ is contained in a small neighborhood of $L$. The invariant line $L$ is called {\em super-attracting, attracting, parabolic} or {\em elliptic} if $0$ is a super-attracting, attracting, parabolic or elliptic fixed point of $p$, respectively. The counter examples constructed in \cite{astorg2016two} are regular polynomial skew products with a parabolic invariant line, and the {\em parabolic implosion technique} is crucial in their construction.
\medskip
\par  It remains an open  problem whether there are wandering Fatou components for regular polynomial skew products with an attracting invariant line. Note that a generic (i.e. open and dense in the parameter space of polynomials of fixed degree) polynomial map on $\mathbb{P}^1$ does {\bf not} have parabolic or elliptic periodic points.  Since the construction of wandering Fatou components in the parabolic case relies heavily on the parabolic implosion technique, one might expect that  in the attracting case there is actually no wandering Fatou component. This would imply that a {\bf generic} regular polynomial skew product does not have wandering Fatou components, which is  satisfactory for the purpose of understanding  dynamics of generic regular polynomial skew products.

\medskip
\par In this paper we  solve this problem when $f$ is {\em unicritical}, i.e. the critical curve of $f$ has a unique {\em transversal intersection} with $L$. Here the critical curve of $f$ is by definition the critical locus of the map $f:\mathbb{C}^2\to \mathbb{C}^2$. Equivalently this means that $f$ can be conjugated to the following form in a neighborhood of $L$:
\begin{equation}\label{uni}
f(z,w)=(\lambda z, w^d+c(z)),
\end{equation}
where $0<|\lambda|<1$, $d\geq 2$ and $c(z)$ is a (non-constant) holomorphic function in a neighborhood $0$. Let $k\geq 1$ be the order of $c$ at $0$, i.e. $c(z)$ has the local expression $c(z)=c(0)+a_k z^k+\text{higher order term}$, with $a_k\neq 0$. By passing to a coordinate change $z\to \alpha z$ we can further assume that
\begin{equation}\label{1.2}
c(z)=c(0)+ z^k+\text{higher order term},
\end{equation} and we choose $0<r_0<1$ such that  for every $z\in B(0,r_0)$,
\begin{equation}\label{eqn:cz-c0}
 |z^k|\geq |c(z)-c(0)-z^k|.
\end{equation}
\par In the rest of the paper we always assume that $f$ has this form unless otherwise stated. Let $f_0$ be the one-dimensional map $f_0(w)=w^d+c(0),$ which is the restriction of $f$ on the invariant line $\left\{z=0\right\}$.
\medskip
\subsection{Main results}
\par The following is our main theorem.
\begin{theorem}\label{main}
Let $f(z,w)=(\lambda z, w^d+c(z))$ as in (\ref{uni}). Then every Fatou component of $f$ is an extension of a Fatou component of $f_0$. In particular there is no wandering Fatou component.
\end{theorem}
Note that the escaping Fatou component $$\Omega_\infty:=\left\{x\in B(0,r_0)\times\mathbb{C}:|f^n(x)|\to +\infty\right\}$$ is clearly non-wandering. Any other Fatou component of $f$ is contained in the complement of $\Omega_\infty$ and hence is uniformly bounded. 
\par We will discuss a partial generalization of Theorem \ref{main} in Section 6 without assuming the unicritical  condition. 
\medskip

Applying Theorem \ref{main} to the setting when $f$ is a globally defined holomorphic endomorphism on $\mathbb{P}^2$, we get the following.
\begin{theorem}\label{main2}
Let $f(z,w)=(p(z), q(z,w))$ be a regular polynomial skew product of degree at least 2. Assume that $p$ does not have parabolic periodic points nor Siegel periodic points, and that for every attracting $p$-periodic point $z_0$ of periods $s$, the critical curve of $f^s$ in $\mathbb{C}^2$ has a unique  transversal intersection with the vertical  line $L:=\left\{z=z_0\right\}\subset \mathbb{C}^2$, then $f$ has no wandering Fatou component in $\mathbb{P}^2$.
\end{theorem}
\medskip

Here is a specific example that Theorem \ref{main2} can apply.
\begin{theorem}\label{main3}
	Let $f(z,w)=(z^2+\lambda z, w^2+az^2+bz+c)$, where $\lambda,a,b,c\in \mathbb{C}$. Then $f$ can be extended holomorphically to $\mathbb{P}^2$. Assume that  $|\lambda|<1$. Then $f$ has no wandering Fatou component in $\mathbb{P}^2$.
\end{theorem}
\medskip

It is not hard to show that every quadratic regular polynomial skew product  can be conjugated to the form $f(z,w)=(z^2+\lambda z, w^2+az^2+bz+c)$, see \cite[Lemma 2.9]{astorg2023hyperbolicity}.

 In the setting of Theorem \ref{main3}, let $f_0(w):=w^2+c$, which is the one-dimensional map acting on the invariant attracting vertical line $\left\{z=0\right\}$. As far as we know, previously the results about no wandering Fatou components in $\mathbb{P}^2$ in the setting of Theorem \ref{main3},  were only established in the following two special cases:
 
 (1) For fixed triple $(a,b,c)\in \mathbb{C}^3$, when $|\lambda|$ is sufficiently small,  \cite{ji2020non}. See also \cite{lilov2004fatou}.

 (2) When $|\lambda|<1 $ and $f_0$ is Collet-Eckmann and Weakly Regular, \cite{ji2019nonuni}. See also \cite{peters2018fatou}.
\medskip
\par In order to prove Theorem \ref{main}, we prove two intermediate results which are of independent interests. The following is the first one concerning the lower bounds of the derivative along the orbit. 
\medskip

Let $f(z,w)=(\lambda z, w^d+c(z))$ as in (\ref{uni}). An orbit $\left\{x_i=(z_i,w_i)\right\}_{0\leq i\leq n}$ of $f$ is called {\em tame} if $|z_i|^k\leq |w_i|^d$ for every $0\leq i< n$ and $|z_0|<r_0$. Here the integer $k$ comes from (\ref{1.2}) and the constant $r_0$ comes from (\ref{eqn:cz-c0}).

\begin{theorem}[Expansion along tame orbits]\label{thm:derivatives}
Let $f(z,w)=(\lambda z, w^d+c(z))$ as in (\ref{uni}). Assume that $f_0$ has no attracting nor super-attracting cycle in $\mathbb{C}$. Given $0<\lambda_0<1$ there exists  $C=C(\lambda_0)>0$  such that if $\left\{x_i=(z_i,w_i)\right\}_{0\leq i\leq n}$ is a tame orbit of $f$,  then
\begin{equation*}
|Df^{n}(x_0)(v)|\geq C \lambda_0^n \min_{i=0}^{n-1} |w_i|^{d-1}.
\end{equation*}
where $v=(0,1)$ is the unit vertical vector. Moreover, if $|w_n|\le |w_j|$ for all $0\le j<n$, then
\begin{equation*}
|Df^{n}(x_0)(v)|\geq C \lambda_0^n.
\end{equation*}

\end{theorem}

\medskip
\par Theorem  \ref{thm:derivatives} is a generalization of the corresponding one-dimensional estimate  obtained by Levin-Przytycki-Shen in \cite{levin2016lyapunov}.
\medskip
\par Our next intermediate result is about the abundance of points satisfying the slow approach condition. For  $\alpha>0$, a point $x_0\in B(0,r_0)\times\mathbb{C}$  is called   {\em $\alpha-$slow approach} (to the critical point),
\begin{equation*}
|w_n|\geq e^{-\alpha n}
\end{equation*}
for every $n$ large, where $w_n$ is the projection of $f^n(x_0)$ to the $w-$coordinate.
\begin{theorem}\label{slowapp}
Let $f(z,w)=(\lambda z, w^d+c(z))$ as in (\ref{uni}). Assume $(0,0)$ is not a periodic point. Then for  every $\alpha>0$, Lebesgue a.e. point in $B(0,r_0)\times \mathbb{C}$ is $\alpha-$slow approach.
\end{theorem}
\par We note that under additional hyperbolicity assumptions on $f_0$, the above theorem was proved in \cite{ji2019nonuni}. Here we do not make any hyperbolicity assumptions on $f_0$. Theorem  \ref{slowapp}  is also a generalization of the corresponding one-dimensional result  obtained by Levin-Przytycki-Shen in \cite{levin2016lyapunov}.

\medskip

Throughout the paper, the vector $v$ will denote the unit vertical tangent vector $v=(0,1)$.

\medskip
\subsection{Previous results} The problem of classification of Fatou components for polynomial skew products gets much attention in recent years. We refer the readers to the survey paper by Dujardin \cite{dujardin2021geometric}.
\medskip
\par As we have mentioned, wandering Fatou components for parabolic polynomial skew products were constructed in \cite{astorg2016two}, see also Astorg-Boc Thaler-Peters \cite{astorg2019wandering}, Hahn-Peters \cite{hahn2021polynomial} and Astorg-Boc Thaler \cite{astorg2022dynamics} for constructions of wandering domains using similar technique.
\medskip
\par In the case that $f$ is an attracting polynomial skew product, previously there are two kinds of results concerning non-existence of wandering Fatou components, both making additional assumptions on $f$. The first type of results assume the smallness of the multiplier $\lambda$. It was first showed by Lilov \cite{lilov2004fatou} that there is no wandering Fatou component if $p'(0)=0$. Later this was generalized by the first named author \cite{ji2020non} that there is no wandering Fatou component if $\lambda$ is sufficiently small. The second type of results make hyperbolicity conditions on $f_0$, see for instance Peters-Vivas \cite{peters2016polynomial}, Peters-Smit \cite{peters2018fatou} and the first named author \cite{ji2019nonuni}. For example the result in \cite{ji2019nonuni} asserts that if $f_0$ satisfies Collet-Eckmann condition and a Weakly Regular condition, then there is no wandering Fatou component and the Julia set has zero volume. Compared with Theorem \ref{main}, here we do not make any assumptions on $\lambda$ nor on the hyperbolicity of $f_0$.
\medskip
\par The elliptic polynomial skew products were studied by Peters-Raissy \cite{peters2019fatou}.  Finally we mention that in the context of complex H\'enon maps, wandering Fatou components were constructed by Berger-Biebler \cite{berger2022emergence}.
\medskip

The topology of Fatou components of polynomial skew products on $\mathbb{C}^2$ was studied by Roeder \cite{MR2769221}.
\medskip
\subsection{On methods of the proof}
  We introduce new methods in order to prove Theorem \ref{main}. Our methods are based on the {\em parameter exclusion technique} and the {\em binding argument}  initiated by Jakobson \cite{jakobson1981absolutely} and Benedicks-Carleson  \cite{benedicks1985iterations}. They successfully used these two techniques to show the existence of absolutely continuous invariant measures for a positive volume parameter set of quadratic interval maps.
  \medskip
  \par As far as we know this is the first time that these two techniques are used to solve the problems about wandering domains.  This is actually built on the following interesting analogue between families of one-dimensional dynamical systems and polynomial skew products. Let $\left\{q_z\right\}_{z\in\mathbb{D}}$ be a family of polynomials of same degree, parametrized by the unit disk. We can associate this family to a polynomial skew product as $$f:\mathbb{D}\times \mathbb{C}\to\mathbb{D}\times \mathbb{C}, f(z,w):=(z,q_z(w)).$$ Recall that a polynomial skew product with an attracting invariant line $L$ has the normal form $$f: U\times \mathbb{C}\to U\times \mathbb{C}, f(z,w)=(\lambda z, q_z(w))$$ in a neighborhood $U\times \mathbb{C}$ of $L$, where $|\lambda|<1$. Now it is obvious to observe the similarity of these two objects. Unlike  excluding bad parameters as Jakobson and Benedicks-Carleson did,  we exclude bad vertical lines instead for attracting polynomial skew products. The one-dimensional estimates in Levin-Przytycki-Shen \cite{levin2016lyapunov} are also involved, to ensure the binding argument work.
   To prove Theorem \ref{thm:derivatives} we use the binding argument  with the help of the one-dimensional estimates in \cite{levin2016lyapunov}. To prove Theorem \ref{slowapp} we use the binding argument and the parameter exclusion technique.

   \par To complete the proof of Theorem \ref{main}, we also adapt some arguments in \cite{lilov2004fatou} and \cite{ji2020non}: assume by contradiction that there is a wandering Fatou component $\Omega$. Then by Theorem \ref{thm:derivatives} and \ref{slowapp} we can construct a vertical disk (i.e. a disk contained in a vertical line) $D\subset\Omega$ such that the forward image $f^n(D)$ contains a vertical disk of radius at least $\lambda_0^n$ when $n$ large enough, where $\lambda_0>0$ is a constant smaller but close to 1. Since $D$ is contained in a wandering Fatou component, the orbit of $D$ cluster only on the Julia set of $f_0$. Again by using the binding argument, we show that this vertical disk actually can not exist. This gives a contradiction hence finishes the proof.
\subsection{Organization of the paper}
\par The organization of the paper is as follows. In Section 2 we do some preliminaries, we recall some one-dimensional results and we introduce the notion of binding time. In section 3 we prove Theorem \ref{thm:derivatives}. In Section 4 we prove Theorem \ref{slowapp}. In Section 5 we complete the proof of Theorem \ref{main}, and we prove Theorem \ref{main2} and \ref{main3}.  Finally in Section 6 we discuss a partial generalization of our results  when $f$ is multicritical.
\medskip
\par {\em Acknowledgements.} The authors are supported by National Key Research and Development Program of China (Grant No 2021YFA1003200). The first named author Zhuchao Ji is supported by ZPNSF grant (No.XHD24A0201) and NSFC Grant (No.12401106). The second named author Weixiao Shen is also supported by the New Cornerstone Foundation through the New Cornerstone Investigator Program and the X'plore prize. 
\bigskip
\section{Preliminaries}
\subsection{Expansion of one-dimensional unicritical polynomials}
In the following $f_0$ will denote a unicritical polynomial
\begin{equation}\label{uni-one}
  f_0(w)=w^d+c, \;\text{where} \;c\in \mathbb{C}, d\geq 2.
\end{equation}
The following proposition summarizes results on lower bounds of derivatives in \cite{levin2016lyapunov}.
\begin{proposition}\label{prop:1-dim-exp}
Let $f_0(w)=w^d+c$ as in (\ref{uni-one}) such that $f_0$ has no attracting nor superattracting cycle in $\mathbb{C}$. Then the following hold:
Given any $0<\lambda_0<1$ there exists $C=C(\lambda_0)>0$ such that for any $w\in \mathbb{C}$ and any $n\in\mathbb{N}$,
\begin{equation}\label{eqn:1-dim-der}
|Df_0^n(w)|\ge C \lambda_0^n \min_{j=0}^{n-1}|f_0^j(w)|^{d-1}.
\end{equation}
More precisely,
\begin{enumerate}
\item[(i)] Given $0<\lambda_0<1$ and $\delta>0$ there exists $\kappa=\kappa(\lambda_0,\delta)>0$ such that for every $w\in\mathbb{C}$ satisfying $|f_0^j(w)|\geq \delta$ for each $0\le j< n$, we have
\begin{equation*}
 |Df_0^n(w)|\geq \kappa\lambda_0^n.
\end{equation*}
\item [(ii)] Given $0<\lambda_0<1$, there exists $\delta_0>0$ such that for any $\delta\in (0,\delta_0]$, if $|w|<\delta$, $|f_0^n(w)|\le \delta$, and $|f_0^j(w)|>\delta$ for all $1\le j<n$, then
    $$|Df_0^n(w)|\ge \lambda_0^n \min \left(1, \left(\frac{|w|}{|f_0^n(w)|}\right)^{d-1}\right).$$
\item [(iii)] Given $0<\lambda_0<1$, there exists $\kappa_0=\kappa_0(\lambda_0)>0$ such that if $|f_0^j(w)|\ge |f_0^n(w)|$ for each $0\le j<n$, then
    \begin{equation*}
  |Df_0^n(w)|\geq \kappa_0\lambda_0^n.
    \end{equation*}
\end{enumerate}
\end{proposition}
\begin{proof} We first prove (i)-(iii).
\begin{enumerate}
\item [(i)] This follows from ~\cite[Lemma 2.2]{levin2016lyapunov}.
\item [(ii)] If $|w|\geq |f_0^n(w)|$, this follows from ~\cite[Lemma 4.2]{levin2016lyapunov}. If $|w|< |f_0^n(w)|$, this follows from  \cite[Lemma 2.1]{levin2016lyapunov}. 
\item [(iii)] Fix $\lambda_0\in (0,1)$, let $\delta_0>0$ be given by (ii) and let $\kappa_0=\kappa(\lambda_0,\delta_0)$ be given by (i).
If $|f_0^n(w)|\ge \delta_0$, then the desired estimate follows from (i). Otherwise, define $\cdots>n_2 >n_1\ge 0$ inductively
such that
\begin{itemize}
\item $n_1$ is  minimal such that $|f_0^{n_1}(w_1)|<\delta_0$,
\item for each $i\ge 1$, $n_{i+1}$ is minimal such that $n_{i+1}>n_i$ and $|f_0^{n_{i+1}}(w)|\le |f_0^{n_i}(w)|.$
\end{itemize}
Since $|f_0^n(w)|\le |f_0^j(w)|$ for all $0\le j<n$, there exists $k$ such that $n_k= n$. So by (i) and (ii)
$$|Df_0^{n}(w)|=|Df_0^{n_1}(w)|\prod_{i=1}^{k-1} |Df_0^{n_{i+1}-n_i}(f_0^{n_i}(w))|\ge \kappa_0 \lambda^n.$$
\end{enumerate}
Let us finally prove (\ref{eqn:1-dim-der}). By (iii), we only need to consider the case $|w|< |f_0^j(w)|$ for each $0< j\le n$. (If $|f_0^s(w)|=\min_{j=0}^{n-1}|f_0^i(w)|$ for some $0<s\leq n-1$, estimate $|Df_0^s(w)|$ and $|Df_0^{n-s}(f^s(w))|$ separately.)  By (i), we may assume $|w|<\delta_0$. Let $s$ be maximal in $\{0,1,\ldots, n\}$ such that $|f_0^s(w)|\le \delta_0$. By (i) again,
$$|Df_0^{n-s}(f_0^s(w))|\ge d|f_0^s(w)|^{d-1}\kappa(\lambda_0,\delta_0) \lambda_0^{n-s-1}.$$
Let $0=s_0<s_1<\cdots<s_p=s$ be such that $s_{i+1}$ is minimal such that $s_{i+1}>s_i$ and $|f_0^{s_{i+1}}(w)|=\min_{j=s_i+1}^n |f_0^j(w)|$. Then by (ii),
    $$|Df_0^s(w)|=\prod_{i=1}^p |Df_0^{s_{i+1}-s_i}(f_0^{s_i}(w))\ge \lambda_0^s \left(\frac{|w|}{|f_0^s(w)|}\right)^{d-1}.$$
Thus
$$|Df_0^n(w)|\ge \kappa(\lambda_0,\delta_0) \lambda_0^n |w|^{d-1}.$$
\end{proof}

\medskip
\par The following proposition was also proved in \cite{levin2016lyapunov}.
\begin{proposition}\label{prop:levin}
Let $f_0(w)=w^d+c$ as in (\ref{uni-one}) such that $f_0$ has no attracting nor superattracting cycle in $\mathbb{C}$. Then
\begin{enumerate}
\item The lower Lyapupov exponent of $f_0$ at $c$ is nonnegative:
$$\liminf_{n\to\infty} \frac{1}{n}\log |Df_0^n(c)|\ge 0,$$
\item The power series
\begin{equation*}
F(z)=1+\sum_{n=1}^{+\infty}\frac{z^n}{(f_0^n)'(c)}
\end{equation*}
has the radius of convergence at least 1, and $F(z)\neq 0$ for every $|z|<1$.
\end{enumerate}
\end{proposition}
\begin{proof} The first statement is one of the main theorems in \cite{levin2016lyapunov}, and the second one is \cite[Corollary 5.1]{levin2016lyapunov} which follows from the first by a result of Levin in \cite{levin2002analytic}.
\end{proof}
\subsection{The binding time}
In this subsection we introduce a version of the notion of {\em binding time} which was first used by Benedicks-Carleson \cite{benedicks1985iterations} to study perturbation of non-uniformly expanding interval maps. Throughout we fix a map $f(z,w)=(\lambda z, w^d+c(z))$ as in (\ref{uni}).

\begin{definition}
Let $x\in B(0,r_0)\times\mathbb{C}$ and $y\in B(0,r_0)\times\mathbb{C}$. Let $\mu>0$ be a constant. For every integer $m\geq 0$ we define $\xi_m(x)$ to be  the projection of $f^m(x)$ to the $w-$coordinate. We define the $\mu$-binding time of the pair $(x,y)$, denoted by $b_{\mu}(x,y)$, to be the infimum of the set of non-negative integer $n$ such that
\begin{equation}\label{eqn:binding0}
|\xi_n(x)-\xi_n(y)|\geq \frac{\mu \min (|\xi_n(x)|,|\xi_n(y)|)}{(n+1)^2}.
\end{equation}
Here, we use the convention that $\inf \emptyset =+\infty$.

\end{definition}

We note that by our definition of binding time, for $m< b_\mu(x,y)$,   $\xi_j(y)$ and $\xi_j(x)$ are both non-zero for $0\leq j< m$, we have  

\begin{equation}\label{eq:binding}
\frac{|\xi_m(x)-\xi_m(y)|}{|\xi_m(y)|}<\frac{\mu}{(m+1)^2}.
\end{equation}

\medskip

We fix a small constant $\mu_0\in (0,1)$ such that
\begin{equation}\label{eqn:mu_0}
\sum_{i=1}^{\infty} \frac{2\mu_0}{i^2}<\frac{1}{4d},
\end{equation}
and such that for any $t\in \mathbb{C}$ with $|t|<\mu_0$, we have 
\begin{equation}\label{eqn:mu0'}
|t^d-1|\le d|t-1|(1+d|t-1|).
\end{equation}

Let 
\begin{equation}\label{eqn:mu1}
\mu_1=\frac{\mu_0}{4}.
\end{equation}
Let us show that for every positive integer $n$,
\begin{align} \label{eqn:2.4}
|w_1-w_2|\le \frac{\mu_1\min (|w_1|, |w_2|)}{n^2},\,\, & |w_2-w_3|\le \frac{\mu_1\min (|w_2|, |w_3|)}{n^2} \notag\\
 & \implies \notag \\ 
|w_1-w_3|& \le \frac{\mu_0\min (|w_1|, |w_3|)}{n^2}.
\end{align}

Indeed, 
\begin{align*}
|w_1-w_3| & \le |w_1-w_2|+|w_2-w_3|\\
&\le \frac{\mu_1\min (|w_1|, |w_2|)}{n^2}+ \frac{\mu_1\min (|w_2|, |w_3|)}{n^2}
\\&\le \frac{2\mu_1|w_2|}{n^2}.
\end{align*}
On the other hand, we have $$|w_2|\leq |w_1|+|w_1-w_2|\leq |w_1|+\frac{\mu_1\min (|w_1|, |w_2|)}{n^2}\leq 2|w_1|.$$
Similarly,  $|w_2|\leq 2|w_3|$. Hence (\ref{eqn:2.4}) holds.

By (\ref{eqn:2.4}),  we have
\begin{equation}\label{eqn:mu0mu1}
b_{\mu_0}(x_1,x_3)\ge \min(b_{\mu_1}(x_1,x_2), b_{\mu_1}(x_2, x_3)).
\end{equation}

We note that (\ref{eqn:mu0mu1}) will only be used in the proof of Lemma \ref{lemma5}.

\par In the rest of the paper a $\mu_0-$binding time is called a binding time for simplicity.
\par Note that if $n$ is the $\mu$-binding time of $(x,y)$, where $0<\mu\leq \mu_0$, then

\begin{equation}\label{eq:binding2}
	\frac{1}{2}|\xi_j(y)|\le |\xi_j(x)|\le 2|\xi_j(y)| \text{ when } 0\le j< n.
\end{equation}

\medskip

The following two lemmas will be used frequently in the rest of the paper.
\begin{lemma}\label{*}
Let $x\in B(0,r_0)\times\mathbb{C}$ and $y\in B(0,r_0)\times\mathbb{C}$. Let $n=b_{\mu}(x,y)$, where $0<\mu\leq \mu_0$. Then for every positive integer $m\leq n$, we have
\begin{equation*}
\left|\frac{Df^m(x)(v)}{Df^m(y)(v)}-1\right|< \frac{1}{2}.
\end{equation*}
\end{lemma}
\begin{proof}
For each $0\le i<n$, by (\ref{eq:binding}) and the discussion just above  (\ref{eq:binding}), we have $\xi_i(x)$ and $\xi_i(y)$ are both non-zero, and 
\begin{equation*}
\left|\log \frac{\xi_i(x)}{\xi_i(y)}\right|\le 2 \frac{|\xi_i(x)-\xi_i(y)|}{|\xi_i(y)|}\le \frac{2\mu}{(i+1)^2}.
\end{equation*}
Thus for $1\le m\le n$,
$$\left|\sum_{i=0}^{m-1} \log\frac{\xi_i(x)}{\xi_i(y)}\right|\le \sum_{i=0}^{m-1} \frac{2\mu}{(i+1)^2}<\frac{\mu_0\pi^2}{3}<\frac{1}{4d}.$$
Consequently,
$$\left|\frac{Df^m(x)(v)}{Df^m(y)(v)}-1\right|=\left|e^{(d-1)\sum_{i=0}^{m-1} \log \frac{\xi_i(x)}{\xi_i(y)}}-1\right|< 2(d-1) \left|\sum_{i=0}^{m-1} \log \frac{\xi_i(x)}{\xi_i(y)}\right|<\frac{1}{2}.$$
\end{proof}
\medskip
\par Let $x=(z_0,w_0)\in B(0,r_0)\times\mathbb{C}$, $y=(z,w)\in B(0,r_0)\times\mathbb{C}$ and $n\geq 1$. We define the following quantity
\begin{equation}\label{eqn:W}
W(x,y,n):=2|w_0-w|+\sum_{i=1}^{n}\frac{2|c(\lambda^{i-1}z_0)-c(\lambda^{i-1}z)|}{|Df^i(x)(v)|},
\end{equation}
where  $c(z)$ is the critical value curve of $f$ as in (\ref{uni}). By our assumption (\ref{1.2}), this implies
\begin{equation}\label{eqn:W00}
W((0,w),(z_0,w_0),n)\leq 2|w_0-w|+\sum_{i=1}^{n}\frac{4|\lambda^{k(i-1)}z_0^k|}{|Df_0^i(w)|}.
\end{equation}

\begin{lemma}\label{**}
Let $f(z,w)=(\lambda z, w^d+c(z))$ as in (\ref{uni}). Let $x=(z_0,w_0)\in B(0,r_0)\times\mathbb{C}$ and $y=(z,w)\in B(0,r_0)\times\mathbb{C}$. If $n$ is a positive integer and $n\leq b_{\mu}(x,y)$, where $0<\mu\leq \mu_0$, then
$$|Df^n(x)(v)|\ge \frac{|\xi_n(x)-\xi_n(y)|}{W(x,y,n)}.$$
In particular if $n=b_{\mu}(x,y)$ then
$$ |Df^n(x)(v)|\geq \frac{\mu\min(|\xi_n(x)|, |\xi_n(y)|)}{ (n+1)^2 W(x,y,n) }.$$ 
\end{lemma}
\begin{proof}
First, for every $1\leq m\leq n$, we have
\begin{align}\label{2.2}
 |\xi_m(x)-\xi_m(y)|&=|(\xi_{m-1}(x))^d-(\xi_{m-1}(y))^d+c(\lambda^{m-1}z_0)-c(\lambda^{m-1}z)| \notag \\
 &\leq |(\xi_{m-1}(x))^d-(\xi_{m-1}(y))^d|+|c(\lambda^{m-1}z_0)-c(\lambda^{m-1}z)|
\end{align}
\medskip
\par We now estimate the first term in (\ref{2.2}). Since 
$$\left|\frac{\xi_{m-1}(y)}{\xi_{m-1}(x)}-1\right|\le \frac{\mu}{m^2}\le \mu_0,$$
by (\ref{eqn:mu0'}), we have 
\begin{multline*}
\frac{|\xi_{m-1}(x)^d-\xi_{m-1}(y)^d|}{|\xi_{m-1}(x)^d|}
\le d \frac{|\xi_{m-1}(x)-\xi_{m-1}(y)|}{|\xi_{m-1}(x)|}\left(1+d \frac{|\xi_{m-1}(x)-\xi_{m-1}(y)|}{|\xi_{m-1}(x)|}\right)
\\
\le d \frac{|\xi_{m-1}(x)-\xi_{m-1}(y)|}{|\xi_{m-1}(x)|}\left(1+\frac{d\mu}{m^2}\right),
\end{multline*}
and hence
$$|\xi_{m-1}(x)^d-\xi_{m-1}(y)^d| \le d|\xi_{m-1}(x)|^{d-1} |\xi_{m-1}(x)-\xi_{m-1}(y)|\left(1+\frac{d\mu}{ m^2}\right).$$

%
Therefore for every $1\leq m\leq n$, we have
\begin{equation*}
\frac{|\xi_m(x)-\xi_m(y)|}{|Df^m(x)(v)|}\leq \left(1+\frac{d\mu}{m^2}\right)\frac{|\xi_{m-1}(x)-\xi_{m-1}(y)|}{|Df^{m-1}(x)(v)|}+\frac{|c(\lambda^{m-1}z_0)-c(\lambda^{m-1}z)|}{|Df^m(x)(v)|}.
\end{equation*}
\medskip
\par Next we estimate $\frac{|\xi_{m-1}(x)-\xi_{m-1}(y)|}{|Df^{m-1}(x)(v)|}$ by the same method. Continue doing this process we finally have
\begin{equation*}
\frac{|\xi_n(x)-\xi_n(y)|}{|Df^n(x)(v)|}\leq 2|w_0-w|+\sum_{i=1}^{n}\frac{2|c(\lambda^{i-1}z_0)-c(\lambda^{i-1}z)|}{|Df^i(x)(v)|}=W(x,y,n),
\end{equation*}
using the fact that $\prod_{m=1}^\infty (1+\frac{d\mu}{m^2})\le \prod_{m=1}^\infty (1+\frac{d\mu_0}{m^2})\le e^{1/8}< 2$.

Hence we have 
\begin{equation*}
|Df^n(x)(v)|\geq \frac{|\xi_n(x)-\xi_n(y)|}{W(x,y,n)}.
\end{equation*}
In particular by (\ref{eqn:binding0}),  when $n=b_\mu(x,y)$ we have
$$|Df^n(x)(v)| \geq \frac{\mu\min (|\xi_n(x)|, |\xi_n(y)|)}{ (n+1)^2 W(x,y,n) }.$$
\end{proof}

\par The following lemma will be used to bound from below the vertical derivatives for an orbit starting close to the critical value line, and will be used in Section 4 and Section 5. 
\begin{lemma}\label{lem:bdcv}
Assume that $f_0$ has no attracting nor superattracting cycle in $\mathbb{C}$. Fix $0<\mu\leq \mu_0$. For any positive integer $s_0$ and $\lambda_0\in (0,1)$ there exists $\delta_0>0$ such that the following holds. Let $(z_j,w_j)_{j=0}^\infty$ be an $f$-orbit that satisfies:
\begin{equation}\label{eqn:dfdelta}
	\delta:=\max(|w_1-c(0)|, |z_0|^k|)^{1/d}\in (0,\delta_0).
\end{equation}
	Then there exists a positive integer $n$ 
	such that the following hold:
\begin{enumerate}
\item $s_0\le n\le b_\mu((z_1,w_1), (0, c(0)) $,
\item	$|w_j|\ge \delta/2, \,\, 1\le j\le n,$
\item
	$|Df^n(z_1,w_1)(v)|\ge \lambda_0^{n} \delta^{-(d-1)}.$
\end{enumerate}
\end{lemma}

\begin{proof}   	

Since $f$ is Lipschitz, we have $m=b_\mu ((z_1,w_1), (0, c(0)))\to +\infty$ when $\delta\to 0$. This implies that  $ m=b_\mu ((z_1,w_1), (0, c(0)))$ can be made arbitrarily large, if $\delta_0>0$  is sufficiently small.

Fix a constant $\lambda_1\in (0,1)$ such that $\lambda_1\ge \max(\lambda_0, |\lambda|^k)$. 
Let $s$ be a large positive integer such that $s\ge s_0$, $\lambda_0^{-s/2}>2d$ and 
\begin{equation}\label{eqn:mlarge}
\frac{\mu}{2C(p+1)^2}\geq \lambda_0^p
\end{equation}
for each $p\ge s$, where 
\begin{equation}\label{eqn:constantCbdcv}
C=2+4\sum_{i=1}^m \frac{|\lambda|^{ki}}{|Df^i(c_0)|}
\end{equation}
is a positive constant. 
Let $\delta_0>0$ be a small constant such that the following hold:
\begin{itemize}
\item $|f_0^j(0)|> \delta_0$ holds for all $1\le j\le s$;
\item $b_\mu((z_1,w_1), (0, c(0))\ge s$ holds when $\delta\in (0,\delta_0)$, where $\delta$ is as in (\ref{eqn:dfdelta}).   
\end{itemize}

Now assume $\delta\in (0,\delta_0)$ and let $m=b_\mu ((z_1,w_1), (0, c(0)))\in \mathbb{N}\cup\{\infty\}$.
We distinguish two cases. 
	
{\em Case 1.} There exists a minimal integer $0\le n\le m$ such that $|f_0^{n+1}(0)|\le \delta$. 
Then $n\ge s\ge s_0$ and $\lambda_0^{-n/2}> 2d$. By (\ref{eq:binding2}), we have  $|w_j|\ge |f^j(0)|/2>\delta/2$ for all $1\le j\le n$. By Lemma~\ref{*}, 
$|Df^n(z_1,w_1)|\ge |Df_0^n(c(0))|/2,$
so it suffices to show 
\begin{equation}\label{eqn:df0n}
|Df_0^n(c(0))|\ge \delta^{-(d-1)} d^{-1}\lambda_0^{n/2}.
\end{equation}
To this end, note that for any $t\in \mathbb{C}$ with $|t|$ very close to $0$, there is $\delta'$ slightly larger than $\delta$, such that 
$|t|<\delta'$, $|f_0^j(t)|>\delta'$ for $1\le j\le n$ and $|t|<|f_0^{n+1}(t)|<\delta'$.  By Proposition~\ref{prop:1-dim-exp}(ii),  we have 
$$|Df_0^{n+1}(t)|\ge \lambda_0^{n/2}(|t|/|f_0^{n+1}(t)|)^{d-1},$$
which implies that 
  $$|Df_0^n(f_0(t))|=\frac{|Df_0^{n+1}(t)|}{|Df_0(t)|}\ge \frac{1}{d}\lambda_0^{n/2} \frac{1}{|f_0^{n+1}(t)|^{d-1}}.$$
Letting $t\to 0$, we obtain (\ref{eqn:df0n}). 

	
	
{\em Case 2.} For each $0\le j\le m$, $|f_0^{j+1}(0)|> \delta$. Then by (\ref{eq:binding2}), $|w_j|\ge |f_0^j(0)|>\delta/2$ for all $1\le j\le m$. 
By Proposition~\ref{prop:levin}, there exists $C_1>0$ such that 
\begin{equation}\label{eqn:Df0c0}
|Df_0^i(c(0))|\ge C_1\lambda_1^i\mbox{ for all }i\ge 1.
\end{equation}

If $m=\infty$, then it suffices to take $n$ to be a large integer such that $n>s_0$ and $C_1\lambda_1^n> 2\delta^{-(d-1)}\lambda_0^n$, since $|Df^n(z_1,w_1)|\ge |Df_0^n(c(0))|/2$.
In the following, we assume $m<\infty$ and take $n=m$. Then $n=m\ge s\ge s_0$, so the property (1) holds.   
As we noted above, the property (2) also holds. Let us prove the property (3). 
By (\ref{eqn:W00}) and (\ref{eqn:dfdelta}), 
\begin{multline*}
W:=W((0,c(0)),(z_1,w_1), m)\le 2|w_1-c(0)|+\sum_{i=1}^m \frac{4|\lambda|^{ki}|z_0|^k}{|Df^i(c(0))|} 
\le C\delta^d,
\end{multline*}
where $C$ is as in (\ref{eqn:constantCbdcv}). 
By Lemma \ref{**},   $$|Df^m(z_1,w_1)(v)|\ge \frac{\mu}{2(m+1)^2} \frac{|f_0^{m+1}(0)|}{W}\ge \frac{\mu}{2(m+1)^2 C}\delta^{-(d-1)},$$
which implies the property (3) by (\ref{eqn:mlarge}) since $m\ge s$. 
%
\end{proof}
%
\subsection{A variant of Przytycki's lemma}
We shall also frequently use the following slight generalization of a lemma due to Przytycki \cite[Lemma 1]{przytycki1993lyapunov}. In the following, the Euclidean norm in $\mathbb{C}^2$ is denoted by $||\cdot||$.
\begin{lemma}\label{p}
Let $f(z,w)=(\lambda z, w^d+c(z))$ as in (\ref{uni}). Assume that $f_0$ has no attracting nor superattracting cycle in $\mathbb{C}$. Then there exist a constant $ C>0 $ such that for every $\varepsilon>0$ and $n\geq 1$,  if $x=(z_0, w_0)\in B(0,r_0)\times\mathbb{C}$ satisfies $|z_0|^k\leq \varepsilon$, $|w_0|\leq \varepsilon$ and $|\xi_n(x)|\leq \varepsilon$, then $n\geq C\log\frac{1}{\varepsilon}$.
\end{lemma}

Recall that the constant $k$ was defined in (\ref{1.2}).
\begin{proof}
It is sufficient to prove the result when $\varepsilon$ is sufficiently small. By \cite[Lemma 1]{przytycki1993lyapunov}, there are constants $\varepsilon_0\in (0,1)$ and $C_0>0$ such that if $w_0\in\mathbb{C}$ and $n\ge 1$ satisfies $|w_0|<\varepsilon$ and $|f_0^n(w_0)|<\varepsilon$ for some $\varepsilon\in (0,\varepsilon_0)$, then $n\ge C_0\log \frac{1}{\varepsilon}$. 


Let $R>0$ be a constant such that for any $y=(z,w)$ with $|z|<r_0$ and $|w|>R$, we have $|\xi_1(y)|>2|w|$.  Thus $|\xi_n(y)|\leq R$ implies $|\xi_m(y)|\leq R$ for every $0\leq m\leq n$. Let $M:=\max(d R^{d-1},2)$. 

Now consider $x=(z_0,w_0)$ with $|z_0|^k\le \varepsilon$, $|w_0|<\varepsilon$ and $|\xi_n(x)|\le \varepsilon$ for $\varepsilon<\min (\varepsilon_0^2/9, R)$. by the choice of $R$, for each $0\le m<n$, $|\xi_m(x)|\le R$. Let $n_1$ be the maximal positive integer such that $n_1\le n$ and such that $|f_0^m(w_0)|<R$ for all $0\le m<n_1$.  Then,  for each $0\le m<n_1$, 
\begin{multline*}
|\xi_{m+1}(x)-f_0^{m+1}(w_0)|=|\xi_m(x)^d-f_0^m(w_0)^d|+|c(z_m)-c(0)|\\
\leq M |\xi_m(x)-f_0^m(w_0)|+2|z_m|^k\leq M|\xi_m(x)-f_0^m(w_0)|+2\varepsilon,
\end{multline*} 
hence
\begin{equation}\label{a}
|\xi_{n_1}(x)-f_0^{n_1}(w_0)|\leq 2M^{n} \varepsilon.
\end{equation}
Now if $M^{n_1}\geq  \sqrt{\frac{1}{\varepsilon}}$ we get the desired estimate for $C=1/2\log M$. If $M^{n}\leq  \sqrt{\frac{1}{\varepsilon}}$, then by (\ref{a}) we have 
$|f_0^{n_1}(w_0)|\leq |\xi_{n_1}(x)|+2M^{n_1} \varepsilon\leq 3\sqrt{\varepsilon}$. Thus we have $|w_0|\leq \varepsilon\le 3\sqrt{\varepsilon}$ and $|f_0^n(w_0)|\leq 3\sqrt{\varepsilon}$, so
$$n\ge n_1\ge C_0\log \frac{1}{3\sqrt{\varepsilon}},$$
which implies the desired estimate for a suitably chosen $C$.
\end{proof}

\bigskip
\section{Lower bounds of derivatives along the orbits}
In this section we prove Theorem \ref{thm:derivatives} which is a perturbed version of Proposition~\ref{prop:1-dim-exp}. To obtain the desired lower bounds for the vertical derivatives $|Df^n(z_0,w_0)(0,1)|$, we decompose the orbit $(z_j,w_j)$ into sub-orbits falling into the following two cases, and do estimates separately:
\begin{itemize}
\item Orbits staying bounded away from the critical set, i.e. the case $\inf_{j=0}^{n-1} |w_j|$ is bounded away from zero. 
\item Orbits corresponds to returns to a small neighborhood of the critical set, i.e. for some small $\delta>0$, $|w_0|, |w_n|\le \delta,$ but $|w_j|>|w_n|$ for $1\le j<n$. 
\end{itemize}

The first case follows from the corresponding non-perturbed result (Proposition~\ref{prop:1-dim-exp} (i) and (iii)) by continuity argument, see Lemmas~\ref{2-dim-(i)} and ~\ref{lem:kappa0}. 
The second case is more delicate, and is derived from Proposition~\ref{prop:1-dim-exp} (ii), using binding argument, see Lemma~\ref{lem:freturn}. The proof of Theorem~\ref{thm:derivatives} is given at the end of this section. 



\medskip
\par 

We start with the lemmas dealing with orbits staying away from the critical sets which follow from Proposition~\ref{prop:1-dim-exp} by a continuity argument.
\begin{lemma}\label{2-dim-(i)}
Let $f(z,w)=(\lambda z, w^d+c(z))$ as in (\ref{uni}) such that $f_0$ has no attracting nor superattracting cycle in $\mathbb{C}$. Given $0<\lambda_0<1$ and $0<\delta<1$ there exists $\kappa=\kappa(\lambda_0,\delta)>0$ and $\eta_1=\eta_1(\lambda_0, \delta)>0$ such that for any $f$-orbit $\{x_j=(z_j,w_j)\}_{i=0}^n$, if $|w_j|\geq \delta$ hold for all $0\le j< n$ and $|z_0|<\eta_1$, then
\begin{equation*}
    |Df^n(x_0)(v)|\geq \kappa\lambda_0^n.
\end{equation*}
\end{lemma}

\begin{proof} By Proposition~\ref{prop:1-dim-exp} (i), there exists $\kappa>0$ depending on $\lambda_0$ and $\delta$ such that for any $f_0$-orbit $(y_j)_{j=0}^n$ with $|y_j|\geq \delta/2$ for all $0\le j<n$, then
$$|Df_0^n(y_0)|\ge 2\kappa \lambda_0^{n/2}.$$
Choose a positive integer $N$ such that $\kappa^2> \lambda_0^N.$

By continuity, there exists $\eta_1>0$ such that for any orbit $\{x_i=(z_i, w_i)\}_{i=0}^{m-1}$, $m\le N$, with $\min_{j=0}^{m-1} |w_j|\geq \delta$ and $|z_0|<\eta_1$, then $\min_{j=0}^{m-1} |f_0^j(w_0)|>\delta/2$ and
$$|Df^m(x_0)(v)|\ge \frac{1}{2} |Df_0^m(w_0)|.$$ Since
$|Df_0^m(w_0)|> 2\kappa \lambda_0^{m/2},$
we conclude $$|Df^m(x_0)(v)|\ge \kappa \lambda_0^{m/2}.$$
In particular, when $m=N$, this gives us $$|Df^N(x_0)(v)|>\lambda_0^{N}.$$

Now let us consider an $f$-orbit $(x_i=(z_i,w_i))_{i=0}^n$ with $|w_i|\ge \delta$ for $0\leq i<n$. Let $n=qN+r$, where $q$ is a non-negative integer and $0\le r<N$. Then by decomposing this orbit into $q$ pieces of length $N$ together with one piece of length less than $r<N$, we obtain
$$|Df^n(x_0)(v)|\ge \kappa\lambda_0^{r/2}\lambda_0^{qN}\ge \kappa \lambda_0^n.$$
\end{proof}

A completely similar argument, using Proposition~\ref{prop:1-dim-exp} (iii) instead of (i), shows the following stronger result in the case $|w_n|$ is small. The difference between Lemma \ref{2-dim-(i)} and Lemma \ref{lem:kappa0} is that in Lemma \ref{lem:kappa0}, $\kappa_0$ does not  depend on $\delta$, but in Lemma  \ref{2-dim-(i)}, $\kappa$ depends on $\delta$.

\begin{lemma}\label{lem:kappa0}
Let $f(z,w)=(\lambda z, w^d+c(z))$ as in (\ref{uni}) such that $f_0$ has no attracting nor superattracting cycle in $\mathbb{C}$. Given $0<\lambda_0<1$, there exists $\kappa_0=\kappa_0(\lambda_0)>0$ such that for any $\delta>0$ there exists $\eta_2=\eta_2(\lambda_0,\delta)>0$ with the following property: If
$\{x_j=(z_j,w_j)\}_{j=0}^n$ is an $f$-orbit with $|w_j|\geq \delta$ hold for $0\le j< n$, $|w_n|\le \delta$ and $|z_0|<\eta_2$, then
\begin{equation*}
    |Df^n(x_0)(v)|\geq \kappa_0\lambda_0^n.
\end{equation*}
\end{lemma}

\begin{proof} We just specify the choice of $\kappa_0$. By Proposition~\ref{prop:1-dim-exp} (iii), there exists $\kappa_0\in (0,1)$ depending on $\lambda_0$ such that for any $f_0$-orbit $(\widetilde{w}_j)_{j=0}^n$, with $|\widetilde{w}_j|\geq |\widetilde{w}_n|$ for all $0\le j<n$, then
$$|Df_0^n(\widetilde{w}_0)|\ge 2\kappa_0 \lambda_0^{n/2}.$$

Arguing in the same way as the previous proof by a continuity argument, we show that for any $n$, there exists $\eta'_n=\eta'_n(\delta)$ such that if $|z_0|<\eta'_n$ and $|w_j|\ge \delta$ for $0\le j<n$ and $|w_n|<\delta$, then \begin{equation}\label{3.3-1}
|Df^n(x_0)(v)|\ge \kappa_0 \lambda_0^n.
\end{equation}
\par On the other hand, by Lemma \ref{2-dim-(i)}, there exists $\kappa=\kappa(\lambda_0,\delta)>0$ and $\eta_1=\eta_1(\lambda,\delta)>0$ such that $$|Df^n(x_0)(v)|\ge \kappa \lambda_0^{n/2},$$ 
provided that $|z_0|<\eta_1$. So there exists $N_1=N_1(\lambda_0,\delta)>0$ such that 
\begin{equation}\label{3.3-2}
	|Df^n(x_0)(v)|>\kappa_0 \lambda_0^n 
\end{equation}provided that $n>N_1$ and $|z_0|<\eta_1$. 
\par Combine the estimates (\ref{3.3-1})  and (\ref{3.3-2}), the lemma holds with $\eta_2:=\min( \eta_1, \eta_1', \cdots, \eta_{N_1}')$, where $(\eta'_n)_{1\leq n\leq N_1}$ are constants appearing before (\ref{3.3-1}).
\end{proof}

We shall use the binding argument to prove a perturbed version of Proposition~\ref{prop:1-dim-exp} (ii). Let us first state the following one-dimensional slightly generalized form of the statement.

\begin{lemma}\label{lem:irret} Given $\lambda_0\in (0,1)$, there exists $\delta_0>0$ such that for any $\delta\in (0,\delta_0)$,
if $(w_i)_{i=0}^n$ is an $f_0$-orbit with $|w_0|< 2\delta, |w_n|< 2\delta$ and $|w_j|> \delta/2$ for all $1\le j<n$, then
$$|Df_0^n(w_0)|\ge \lambda_0^{n}\min\left(1, (|w_0|/|w_n|)^{d-1}\right).$$
\end{lemma}

We note that the condition on the intermediate
points of the orbit $|w_j|> \delta/2$  is weaker than what it would have been for Proposition
\ref{prop:1-dim-exp} (ii) (where it would be $|w_j|> 2\delta$).
\begin{proof} Fix $\lambda_0\in (0,1)$. By Proposition~\ref{prop:1-dim-exp} (ii), there exists $\delta_0>0$ such that the following holds: If $\{w_j\}_{j=0}^n$ is an $f_0$-orbit for which there exists $\rho\in (0,\delta_0)$ with $|w_0|< \rho, |w_n|\le \rho$, and $\min_{j=1}^{n-1} |w_j|> \rho$, then
\begin{equation}\label{eqn:irret1}
|Df_0^n(w_0)|\ge \lambda_0^{n/2}\min (1, (|w_0|/|w_n|)^{d-1}).
\end{equation}
Let $$N(\delta)=\inf\{s\ge 1: f_0^s(\overline{B(0, 2\delta)})\cap \overline{B(0,2\delta)}\not=\emptyset\}.$$
Then $N(\delta)\to\infty$ as $\delta\to 0$.

Now let $(w_j)_{j=0}^n$ be as in the lemma with $\delta>0$ small.

{\em Case 1.} $|w_0|, |w_n|< \delta/2$. In this case the desired estimate follows immediately from (\ref{eqn:irret1}) by taking $\rho=\delta/2$.

{\em Case 2.} $|w_0|\ge \delta/2$. Let $0=n_0<n_1<n_2<\cdots<n_p=n$ be all the  positive integers in $\{0,1,\ldots, n\}$ with $|w_{n_j}|\le 2\delta$. For each $0\le j<p$, $|w_{n_i}|/|w_{n_{i+1}}|\ge 1/4$, so taking $\rho=2\delta$ in (\ref{eqn:irret1}), we obtain
$$|Df_0^{n_{j+1}-n_j}(w_{n_j})|\ge \lambda_0^{n/2} 4^{-(d-1)}.$$
Thus $$|Df_0^{n} (w_0)|\ge \lambda_0^{n/2} 4^{-p(d-1)}.$$
Since $n_{j+1}-n_j\ge N(\delta)$, $p\le n/N(\delta)$. Provided that $\delta$ is small enough, $p\ll n$, so that $4^{-p(d-1)}\ge \lambda_0^{n/2}$.
Then $|Df_0^n(w_0)|\ge \lambda_0^n$.

{\em Case 3.} $|w_0|< \delta/2$ and $|w_n|\ge \delta/2$. Define $n_j$ as above. Then
$$|Df_0^{n_1}(w_0)|\ge \lambda_0^{n/2} (|w_0|/(2\delta))^{d-1}\ge 4^{-(d-1)} \lambda_0^{n/2} (|w_0|/|w_n|)^{d-1},$$
and $$|Df^{n_{i+1}-n_i}(w_{n_i})|\ge \lambda_0^{n/2} 4^{-(d-1)}.$$
The desired estimate follows similarly as in Case (2).
\end{proof}
Recall the notion of tame orbits was introduced in Section 1, before the statement of Theorem \ref{thm:derivatives}.
\begin{lemma}\label{lem:freturn} Given $0<\lambda_0<1$, there exists $\delta_0>0$ and $\eta_0>0$ such that the following holds. Let $\{x_i=(z_i,w_i)\}_{i=0}^n$ be a tame $f$-orbit such that 
\begin{itemize}
\item $|w_0|, |w_n|\le \delta_0$;
\item $|z_0|<\eta_0$;
\item $|w_j|\ge |w_n|$ for each $0<j<n.$
\end{itemize}
Then
\begin{equation}\label{eqn:perturbedret}
|Df^n(x_0)(v)|\ge \lambda_0^n \min\left(1, (|w_0|/|w_n|)^{d-1}\right).
\end{equation}
%
\end{lemma}
\begin{proof} Without loss of generality, we may assume that $1>\lambda_0>|\lambda|>0$. Let $\kappa_0$ be the constant given by Lemma~\ref{lem:kappa0}. 
Let $N$ be a large positive integer such that 
\begin{equation}\label{eqn:choiceN0}
\lambda_0^{-N}>\max(4^d, 2\kappa_0^{-1}),
\end{equation}
and such that 
\begin{equation}\label{eqn:choiceN}
\frac{\mu_0}{2(s+1)^2C_1}>\lambda_0^s
\end{equation}
holds for all $s\ge N$, where $C_1=C_1(\lambda_0)>0$ is a constant as in (\ref{eqn:C1}).  

Let $\delta_0>0$ be a small constant such that the following holds: 
\begin{itemize}
\item for any orbit $(z'_j,w'_j)_{j=0}^n$ satisfying $|z'_0|\le \delta_0$ and $|w'_0|\le \delta_0$ and $|w'_n|\le \delta_0$, we have $n\ge N$.  
\item if $|z'_0|\le \delta_0, |w'_0|\le \delta_0$, we have $b_{\mu_0} ((z'_0,w'_0), (0,w'_0))>N.$ 
\end{itemize}
Note that for small $\delta_0$, the first property is guaranteed by the variant of Przytycki’s lemma (=Lemma~\ref{p}), and the second property follows from definition of binding time by continuity since the critical point of $f_0$ is not periodic. 
 Let $\eta_2=\eta_2(\lambda_0^{1/2},\delta_0)$ be given by Lemma~\ref{lem:kappa0} and let $\eta_0=\min(\eta_2, \delta_0)$. 

We shall prove this lemma by induction on $n$.

{\bf Starting step.} We take the trivial case $n=0$ as the starting step of the induction.

{\bf Inductive step.} Now let $n_0$ be a positive integer and assume that the lemma holds under the additional assumption that $n<n_0$. Let us consider an $f$-orbit $(w_j)_{j=0}^n$ satisfying the assumption of the lemma with $n=n_0$. Note that existence of such an orbit implies that $n_0\ge N$.

Let $s$ be the $\mu_0$-binding time of $(z_0, w_0)$ and $(0,w_0)$. Then $s\ge N$.
Let $w_{i, 0}=f_0^i(w_0)$. By Lemma~\ref{*}, for any integer $1\le i\le s$,
\begin{equation}\label{eqn:compder}
|Df^i(x_0)(v)|\ge \frac{1}{2} |Df_0^i(w_0)|.
\end{equation}

{\bf Case 1.} $s>n$. Then by Lemma~\ref{lem:irret}, $$|Df_0^n(w_0)|\ge \lambda_0^{n/2} \min (1, |w_0|/|w_n|)^{d-1},$$
which together with (\ref{eqn:compder}) implies the desired estimate, since $$\lambda_0^{-n/2}\ge \lambda_0^{-N/2}>2.$$

{\bf Case 2.} $s\le n$ and there exists $t\in \{1,2,\cdots, s-1\}$ such that $|w_t|\le |w_0|$. Let $t$ be minimal with the last property.
Then $|w_{j,0}|>|w_j|/2\ge |w_0|/2$ for each $0<j<t$ and $|w_{t,0}|<2|w_t|\le 2 |w_0|$. Since $t\ge N$, $\lambda_0^{-t/2}> 2^d$. So by Lemma~\ref{lem:irret},
$$|Df_0^t(w_0)|\ge \lambda_0^{t/2} \min (1, (|w_0|/|w_{t,0}|)^{d-1}\ge 2\lambda_0^t \min (1, |w_0|/|w_t|)^{d-1}\ge 2\lambda_0^t,$$
hence $|Df^t(x_0)(v)|\ge \lambda_0^t$ by (\ref{eqn:compder}). By induction hypothesis, $$|Df^{n-t}(x_t)(v)|\ge \lambda_0^{n-t}\min (1, |w_t|/|w_n|)^{d-1}\ge \lambda_0^{n-t}.$$
Thus the desired estimate holds.

{\bf Case 3.} $s\le n$, and for each $1\le j<s$, $|w_j|>|w_0|$. In this case, we first prove the following inequality
\begin{equation}\label{eqn:estats}
|Df^s(x_0)(v)|\ge \frac{\lambda_0^{s/2}}{2}.
\end{equation}

If $|w_{s,0}|<|w_0|$, then by Lemma~\ref{lem:irret},
$$|Df_0^s(w_0)|\ge \lambda_0^{s/2},$$
so (\ref{eqn:estats}) holds by (\ref{eqn:compder}).
From now on, we assume that 
\begin{equation}\label{eqn:ws0w0}
|w_{s,0}|\ge |w_0|.
\end{equation} 
Let $$W=W((z_0,w_0), (0,w_0),s).$$ By Lemma~\ref{**}, we have 
\begin{equation}\label{eqn:Dfsx0v}
|Df^s(x_0)(v)|\ge \frac{|w_s-w_{s,0}|}{W}.
\end{equation}
Let us show 
\begin{equation}\label{eqn:wsws0}
|w_s-w_{s,0}|\ge \frac{\mu_0}{(s+1)^2} \min(|w_s|,|w_{s,0}|)\ge \frac{\mu_0}{2(s+1)^2}|w_0|.
\end{equation}
Indeed, if $|w_{s}|<|w_0|/2$, then by (\ref{eqn:ws0w0}), $|w_s-w_{s,0}|\ge |w_{s_0}|-|w_s|\ge |w_0|/2$, so (\ref{eqn:wsws0}) holds; if $|w_s|\ge |w_0|/2$, then 
$\min(|w_s|, |w_{s,0}|)\ge |w_0|/2$ and 
(\ref{eqn:wsws0}) follows from (\ref{eqn:binding0}).
Now, let us provide an upper bound for $W$. By Proposition~\ref{prop:1-dim-exp},
$|Df_0^j(w_0)|\ge C |w_0|^{d-1} \lambda_0^{j}$ for all $1\le j\le s$, where $C>0$ is a constant depending only on $\lambda_0$. 
Therefore, by (\ref{eqn:W00}),
\begin{equation}\label{eqn:W''}
W\le \sum_{i=1}^s\frac{4|\lambda^{k(i-1)}|z_0|^k|}{|Df_0^i(w_0)|}\le \sum_{i=1}^\infty \frac{4\lambda_0^{k(i-1)}|z_0|^k}{C|w_0|^{d-1}\lambda_0^i}\le C_1 |w_0|,
\end{equation}
where 
\begin{equation}\label{eqn:C1}
C_1=\frac{4}{C\lambda_0^k (1-\lambda_0^{k-1})}.
\end{equation}
and we have used the assumption $|w_0|\ge |z_0|^{k/d}$ for the last inequality. 
Substituting (\ref{eqn:wsws0}) and (\ref{eqn:W''}) into (\ref{eqn:Dfsx0v}), we obtain 
$$|Df^{s}(x_0)(v)|\ge \frac{1}{2(s+1)^2 C_1},$$
which implies (\ref{eqn:estats}) by (\ref{eqn:choiceN}), since $s\ge N$. 
%

To complete the proof in Case 3, 
let $n_1\ge s$ be minimal such that $|w_{n_1}|<\delta_0$. Then by Lemma~\ref{lem:kappa0},
$$|Df^{n_1-s}(x_s)(v)|\ge \kappa_0 \lambda_0^{n_1-s}.$$
By induction hypothesis,
$$|Df^{n-n_1}(x_{n_1})(v)|\ge \lambda_0^{n-n_1}.$$
Therefore
$$|Df^n(x_0)(v)|\ge \kappa_0\lambda_0^{-s/2}\lambda_0^n/2>\lambda_0^n,$$
where we have used (\ref{eqn:choiceN0}) and $s\ge N$ for the last inequality. 
\end{proof}
\medskip
\par Now we can prove Theorem \ref{thm:derivatives}.
\begin{proof}[Proof of Theorem~\ref{thm:derivatives}] Fix $\lambda_0\in (0,1)$ and let $\delta_0,\eta_0$ be given by Lemma~\ref{lem:freturn}. Let $\eta_1=\eta_1(\lambda_0,\delta_0)$ and $\eta_2=\eta_2(\lambda_0, \delta_0)$ be given by Lemmas~\ref{2-dim-(i)} and~\ref{lem:kappa0} respectively. Let $\eta:=\min(\eta_0, \eta_1, \eta_2)$.

Now let us consider a tame orbit $(x_i=(z_i,w_i))_{i=0}^n$. We may assume that $|z_0|<\eta$. Indeed, if $|z_0|\geq \eta$, let $N\geq 0$ be the minimal integer such that $|z_j|\geq \eta$ for $0\leq j\leq N$.  Then $\eta\le |z_0||\lambda|^N<|\lambda|^N$, so $N$ is bounded. By the tameness assumption, $\inf_{i=0}^{N}|w_i|$ is bounded away from zero, then
$|Df^{N+1}(x_0)(v)|$ is bounded away from zero.

If $|w_j|\ge \delta_0$ holds for all $0\le j<n$, then the desired estimate follows from Lemma~\ref{2-dim-(i)}. So we may assume that there is a minimal $n_0\in \{0,1,\ldots, n-1\}$ such that $|w_{n_0}|<\delta_0$. By Lemma~\ref{lem:kappa0}, $$|Df^{n_0}(x_0)(v)|\ge \kappa_0 \lambda_0^{n_0}.$$
Let $n_1\in \{n_0, n_0+1, \ldots, n\}$ be maximal such that $|w_{n_1}|=\inf_{j=n_0}^n |w_j|$. Then by Lemma~\ref{lem:freturn}, $$|Df^{n_1-n_0}(x_{n_0})(v)|\ge \lambda_0^{n_1-n_0}.$$

If $|w_n|\le |w_j|$ for all $0\le j<n$, then $n_1=n$, so the proof is completed. 
To deal with the general case, 
let us define a sequence of integers $n_1<n_2<\cdots<n_t\le n$ such that
\begin{itemize}
\item for each $1\le i<t$, $|w_{n_{i+1}}|=\inf_{j=n_i+1}^n |w_j|<\delta_0$;
\item for each $n_t+1<j<n$, $|w_j|\ge \delta_0$.
\end{itemize}
Then $|w_{n_1}|\le |w_{n_2}|\le \cdots\le |w_{n_t}|.$
By Lemma~\ref{lem:freturn}, for each $1\le i<t$,
$$|Df^{n_{i+1}-n_i}(x_{n_i})(v)|\ge \lambda_0^{n_{i+1}-n_i} \left(\frac{|w_{n_{i}}|}{|w_{n_{i+1}}|}\right)^{d-1}.$$
By Lemma~\ref{2-dim-(i)},
$$|Df^{n-n_t}(x_{n_t})(v)|\ge d|w_{n_t}|^{d-1} \kappa \lambda_0^{n-n_t-1}.$$
Combining all the displayed inequalities, we obtain the desired estimates.
\end{proof}

\section{Slow approach to the critical point}
In this section we prove Theorem \ref{slowapp}. This theorem will be deduced from the following Theorem~\ref{slow} which asserts that for almost every $z$ with respect to the Lebesgue measure in $B(0,r_0)$, pull-backs of the ball $B(0, e^{-\alpha n})$ along the vertical mappings $w\mapsto \xi_n(z,w)$ have bounded criticality, uniformly in $n$. Let $L_z=\{z\}\times\mathbb{C}$ denote the vertical line passing through $(z,0)$.

\begin{definition}\label{the set}
For  every $\alpha>0$, let $\Lambda_\alpha$ be the subset of  $B(0,r_0)$ characterized by the following property: for every $z\in \Lambda_\alpha$, there is a constant $N=N(z)>0$ such that for every integer $n\geq 1$, for every connected component $V$ of $\xi_n^{-1}(B(0,e^{-\alpha n}))\cap L_z$, there are at most $N$ integers $0\leq m\leq n$ satisfying $0\in f^m(V)$.
\end{definition}
\medskip
\par The following theorem implies Theorem \ref{slowapp}.
\begin{theorem}\label{slow}
Let $f(z,w)=(\lambda z, w^d+c(z))$ as in (\ref{uni}). Assume $(0,0)$ is not a periodic point.  Then for any $\alpha>0$, $\Lambda_\alpha$ is a full Lebesgue measure subset of $B(0,r_0)$.
\end{theorem}
\begin{proof}[Proof of Theorem \ref{slowapp} assuming Theorem \ref{slow}] 
By our assumption, there is $R>1$ such that for any $(z,w)$ in the domian of $f$ with $|w|>R$, we have $|\xi_1 (z,w)|\ge |w|>R$.

Since $\Lambda_\alpha$ has full area, it is sufficient  to prove for every $z\in \Lambda_\alpha$, a.e. $x\in L_z$ (with respect to the Lebesgue measure) is $2\alpha-$slow approach. Let $$E_{n,2\alpha}=\left\{ x\in L_z: |\xi_n(x)|<e^{-2\alpha n}\right\}$$ and $$E_{n,\alpha}=\left\{ x\in L_z: |\xi_n(x)|<e^{-\alpha n}\right\}.$$Let $V$ be a connected component of $E_{n,\alpha}$. By our assumption on $z$, the map $\xi_n:V\to B(0,e^{-\alpha n})$ has degree at most $d^N$. Since $\xi_n(V\cap E_{n,2\alpha})\subset B(0,e^{-2\alpha n})$, it follows from a version of Koebe distortion theorem for multivalent maps, see for instance \cite[Lemma 2.1]{przytycki1998porosity}, that there exists $\alpha'=\alpha'(\alpha, N)>0$ such that
\begin{equation*}
\frac{\text{Vol}(V\cap E_{n,2\alpha})}{\text{Vol}(V)}\leq e^{-\alpha' n}.
\end{equation*}
\medskip
\par Since $\xi_n(V\cap E_{n,2\alpha})\subset B(0,e^{-2\alpha n})\subset B(0, R)$, $V\subset \left\{ z\right\}\times B(0,R)$. 
Thus $\text{Vol}(E_{n,2\alpha})$ is exponentially small with respect to $n$ and thus $\sum_{n=1}^{+\infty} \text{Vol}(E_{n,2\alpha})<+\infty$. By the Borel-Cantelli lemma, Lebesgue a.e. $x\in L_z$ is contained in only finitely many $E_{n,2\alpha}$. This is equivalent to say that Lebesgue a.e. $x\in L_z$ is $2\alpha-$slow approach.
\end{proof}

The rest of this section is devoted to prove Theorem  \ref{slow}.  We shall use the following result which is a special case of  \cite[Lemma 4.5]{ji2019nonuni}.

\begin{lemma}\label{pre}
Let $f(z,w)=(\lambda z, w^d+c(z))$ be as in (\ref{uni}). Assume $(0,0)$ is not a periodic point. Then for  every $\alpha>0$, $0<\beta<1$, there is a constant $N=N(\alpha, \beta)>0$ such that for every integer $n\geq 1$ and  for every $z\in B(0,r_0)$,  for every connected component $V$ of $\xi_n^{-1}(B(0,e^{-\alpha n}))\cap L_z$, there are at most $N$ integers $\beta n\leq m\leq n$ satisfying $0\in \xi_m(V)$.
\end{lemma}

To deduce Theorem~\ref{slow} from Lemma~\ref{pre}, we shall show that for almost every $z$, $\xi_n(\lambda^m z,0)$ cannot be too close to zero. More precisely, we shall show that for almost every $z$, $\lambda^m z$ belongs to the set $\Omega_{\alpha,m}$ defined below. 
\begin{definition}
For each $\alpha>0$ and each integer $m\geq 0$,  define
\begin{equation*}
\Omega_{\alpha,m}=\left\{ z\in B(0,|\lambda|^{m}r_0): |\xi_{n}(z,0)|> |z|^{k/d}e^{-\alpha n} \;\text{for every} \;n\geq 1\right\}.
\end{equation*}
For each integer $l\geq 1$, define  $K_{\alpha,m,l}$ to be the following set 
\begin{equation*}
\left\{ z\in B(0,|\lambda|^{m}r_0):\;l\geq 1\; \text{is minimal such that}\;|\xi_{l}(z,0)|\leq |z|^{k/d}e^{-\alpha l}\right\}.
\end{equation*}

\end{definition}
\medskip
\par We have the following relations between $\Lambda_{\alpha}$, $\Omega_{\alpha,m}$ and $K_{\alpha,m,l}$.
\begin{lemma}\label{set}
Assume $(0,0)$ is not a periodic point of $f$. Then for  every $\alpha>0$, $m\geq 0$ and $k\geq 1$, the following  hold:
\begin{enumerate}
\item For every integer $N\geq 0$, $\bigcap_{m=N}^{\infty} \lambda^{-m}\Omega_{\alpha,m}\subset \Lambda_{2\alpha}\cup\{0\}$.
\item $B(0,|\lambda|^mr_0)\setminus \Omega_{\alpha,m}=\bigcup_{l=1}^\infty K_{\alpha,m,l}$.
\end{enumerate}
\end{lemma}
 Here for $\theta\in\mathbb{C}$ and $X\subset \mathbb{C}$, $\theta X:=\left\{\theta x:x\in X\right\}$.
\begin{proof}
The second statement is obvious by the definition of $\Omega_{\alpha,m}$ and $K_{\alpha,m,l}$. We only prove (1). Let $0\not=z\in \bigcap_{m=N}^{\infty} \lambda^{-m}\Omega_{\alpha,m}$. Then for every $m\geq N$, $n>m$, we have $|\xi_{n-m}(\lambda^m z,0 )|> |\lambda^{km/d}||z|^{k/d}e^{-\alpha (n-m)}$. Thus when $|\lambda^{km/d}|\geq e^{-\alpha n/3}$ and $|z|^{k/d}\ge e^{-\alpha n/3}$, i.e. when $m\leq \alpha d n/(-3k\log|\lambda|)$, and $n$ is large enough, we have that $|\xi_{n-m}(\lambda^m z,0)|> e^{-2\alpha n}$. This in particular implies that for $n$ large enough, for every connected component $V$ of $\xi_n^{-1}(B(0,e^{-2\alpha n}))\cap L_z$, if $m\leq \alpha d n/(-3k\log|\lambda|)$,  then $0\notin f^m(V)$. 
\par On the other hand, by Lemma \ref{pre}, we know that  for every connected component $V$ of $\xi_n^{-1}(B(0,e^{-2\alpha n}))\cap L_z$, there are at most $N(2\alpha, d\alpha /(-3k\log|\lambda|)$ integers $m$ such that $0\in f^m(V)$ for $m> \alpha d n/(-3k\log|\lambda|)$. To summarize, for every connected component $V$ of $\xi_n^{-1}(B(0,e^{-2\alpha n}))\cap L_z$,  there are at most $N(2\alpha, d\alpha /(-3k\log|\lambda|)$ integers $m$ such that $0\in f^m(V)$. Hence $z\in \Lambda_{2\alpha}$. This implies (1).
\end{proof}

To complete the proof, we shall show that $K_{\alpha,m,l}$ has exponentially small area relative to $B(0,|\lambda|^m r_0)$. To this end, we shall analyze the property of the map $z\mapsto \xi_l(z,0)$. We shall show that for each $z_0\in K_{\alpha,m,l}$, $\xi_l(z,0)$ maps a neighborhood of $z_0$ conformally onto its image which contains a ball much larger than $B(0, |z_0|^{k/d}e^{-\alpha l})$, see Lemma~\ref{lem:xinz}.  The strategy bears strong analogue with the parameter exclusion technique introduced by Benedicks-Carleson  in \cite{benedicks1985iterations}. Indeed, the vertical fibres $w\mapsto \xi_l(z,w)$ is a composition of maps close to $f_0$, parameterized in $z$. Provided that $\xi_j(z_0, 0)$ is not too close to $0$ for $1\le j<l$, we can relate the derivative $\frac{\partial \xi_l(z,0)}{\partial z}$ with $Df_0^{l-1}(c(0))$ in a very precise manner, see Lemma~\ref{relation} below.
%
%
%

\newcommand{\X}{\mathcal{X}}

Let
$$\X_l(z)=\frac{d}{dz} \xi_l(z,0).$$
Let
$$\widehat{K}_{\alpha, m,l}=\{z\in B(0, 2|\lambda|^{m}r_0): 4|\xi_{j}(z,0)|\ge |z|^{k/d} e^{-\alpha j}, 1\le j<l\}.$$
Recall that by Proposition \ref{prop:levin} (2),
$$X_0:=\sum_{i=0}^{\infty}\frac{\lambda^{ik}}{(f_0^i)'(c(0))}\not=0.$$

\begin{lemma}\label{relation}
Let $f(z,w)=(\lambda z, w^d+c(z))$ be as in (\ref{uni}). Assume $(0,0)$ is not a periodic point and let $\alpha>0$ be such that $e^{d\alpha} |\lambda|^k<1$. Then for $m$ and $l$ large and for every $z_0\in \widehat{K}_{\alpha,m,l}$, we have
\begin{equation*}
\bigg|\frac{\X_l(z_0)}{Df^{l-1}(f(z_0,0))(v)}-kX_0z_0^{k-1}\bigg|\leq \frac{k|X_0||z_0|^{k-1}}{2}.
\end{equation*}
\end{lemma}
\begin{proof} Fix $1>\lambda_0>|\lambda|$. For $z_0\in \widehat{K}_{\alpha, m,l}$, write $x_i=(z_i,w_i)=f^i(z_0,0)$.

Let us first relate the derivative $\X_l(z_0)$ to the derivative $|Df^{l-1}(x_1)(v)|$. For each positive integer $j$,
\begin{equation*}
\xi_j(z,0)=(\xi_{j-1}(z,0))^d+c(\lambda^{j-1}z).
\end{equation*}
Taking derivatives on both side, we have
\begin{equation*}
\X_j(z)=d(\xi_{j-1}(z,0))^{d-1}\X_{j-1}(z)+\lambda^{j-1}c'(\lambda^{j-1}z).
\end{equation*}
Thus
\begin{equation*}
\frac{\X_j(z_0)}{Df^{j-1}(x_1)(v)}=\frac{\X_{j-1}(z_0)}{Df^{j-2}(x_1)(v)}+\frac{\lambda^{j-1}c'(\lambda^{j-1}z_0)}{Df^{j-1}(x_1)(v)}.
\end{equation*}
As $\X_0=0$, by induction, we obtain
\begin{equation}\label{4.4}
\frac{\X_l(z_0)}{Df^{l-1}(x_1)(v)}=\sum_{i=0}^{l-1}\frac{\lambda^{i}c'(\lambda^{i}z_0)}{Df^{i}(x_1)(v)}.
\end{equation}
By the definition of $k$ (see (\ref{eqn:cz-c0})), there exists a constant $A>0$ such that  $$|c'(z)-kz^{k-1}|\le A|z|^k$$ for all $|z|<|\lambda|r_0$.
Let $$X:=\sum_{i=0}^{l-1} \frac{\lambda^{ik}}{Df^i(x_1)(v)}.$$
Then we have
\begin{align}\label{eqn:xl}
\left|\frac{\X_l(z_0)}{Df^{l-1}(x_1)(v)}-kz_0^{k-1}X\right|&\le \sum_{i=0}^{l-1}\frac{|\lambda^i||c(\lambda^ iz_0)-k(\lambda^i z_0)^{k-1}|}{|Df^i(x_1)(v)|}\\&\le A\sum_{i=0}^{l-1}\frac{|\lambda|^{(k+1)i}}{|Df^i(x_1)(v)|} |z_0|^k . \notag
\end{align}

Let us now provide lower bounds for $|Df^i(x_1)(v)|$. Let $s_0$ be a large integer such that $4^d (e^{\alpha d} |\lambda|^k)^{s_0}<1.$
Let $s$ be the $\mu_0$-binding time of the pair  $(x_1, (z_1, c(0))$. 
Put 
$$\delta=\max(|w_1-c(0)|, |z_0|^k)^{1/d}=\max(|c(z_0)-c(0)|, |z_0|^k)^{1/d}=|z_0|^{k/d}.$$
Provided that $m$ is large enough, $\delta$ is small. So by Lemma~\ref{lem:bdcv}, 
there exists $s_0\le n\le s$, such that 
\begin{equation}\label{eqn:lemma2.6}
|Df^n(x_1)(v)|\ge |z_0|^{-k(d-1)/d} e^{-\alpha n}.
\end{equation}
Since $z_0\in \widehat{K}_{\alpha,m,l}$, for each $n\le i<l$, we have
$$|w_i|^d \ge  |z|^k e^{-\alpha i d} 4^{-d}> |z|^k |\lambda|^{ik} =|z_i|^k,$$
so $(x_i)_{i=n}^l$ is tame. By Theorem~\ref{thm:derivatives}, there exists $C=C(\alpha)>0$, such that for each $n<i\le l$, 
$$|Df^{i-n}(x_n)(v)|\ge C e^{-\alpha (i-n) }\min_{j=n}^i |w_j|^{d-1}.$$
As $|w_j|\ge |z_0|^{k/d} e^{-\alpha j}/4$ for each $1\le j<l$, we obtain   
$$|Df^{i-n}(x_n)(v)|\ge C' |z_0|^{k(d-1)/d} e^{-\alpha (di-n)},$$
where $C'$ is a constant. 
Together with (\ref{eqn:lemma2.6}), this implies 
\begin{equation}\label{eqn:dfix1}
|Df^i(x_1)(v)|=|Df^n(x_0)(v)||Df^{i-n}(x_n)(v)|\ge  C' e^{-d \alpha i}
\end{equation}
for each $n\le i<l$. 
On the other hand, by Lemma~\ref{*}, for each $0\le i\le s$,
$$|Df^i(x_1)(v)|\ge \frac{1}{2} |Df_0^i(c(0))|.$$
Combining with Proposition~\ref{prop:levin}, this implies that (\ref{eqn:dfix1}) remain true for $1\le i<n$ (replacing $C'$ by a larger constant if necessary).
It follows that
$$\sum_{i=0}^{l-1}\frac{|\lambda|^{(k+1)i}}{|Df^i(x_1)(v)|}$$ is bounded from above by a constant.
These lower bounds also imply that
$$|X-X_0|\le \frac{1}{4},$$
provided that $m, l$ are large enough. By (\ref{eqn:xl}), we obtain the desired estimate.
%
%
\end{proof}
\medskip

Consider two sets of positive real numbers $\left\{a_i\right\}_{i\in I}$, $\left\{b_i\right\}_{i\in I}$, where $I$ is an index set. In the following lemma we use the asymptotic notions $\asymp$ and $\succeq$. We say $a_i\asymp b_i$ if there exists a constant $C>0$ such that $b_i/C\leq a_i\leq Cb_i$, for every $i\in I$. We say $a_i\succeq b_i$ if there exists a constant $C>0$ such that $a_i\geq Cb_i$, for every $i\in I$. 
\begin{lemma}\label{lem:xinz} There exist constants $C>0$ and $\omega\in (0,1)$ such that for any $\alpha>0$ with $e^{d\alpha} |\lambda|^k<1$, the following holds provided that $m,l$ are sufficiently large. For each $z_0\in K_{\alpha,m,l}$, there exists a neighborhood $V'$ of $z_0$ such that
$\varphi_l(z)=\xi_l(z,0)$ maps $V'$ conformally onto $B(0, C |z_0|^{k/d} e^{-\alpha\omega l})$ and such that $V'\subset B(0, 2|z_0|)$.
\end{lemma}
\begin{proof} We shall prove that there is $r_1\in (0, |z_0|)$ such that $\varphi_l(z):=\xi_l(z,0)$ is univalent on $B(z_0, r_1)$ and $\varphi_l(B(z_0,r_1))$ contains a ball centered at $\xi_l(z_0,0)$ with radius at least $C |z_0|^{k/d} e^{-\alpha \omega l}$, where $C>0$ and $\omega\in (0,1)$ are constants.

Note that $0\not\in K_{\alpha,m,l}$, so $z_0\not=0$.
Let $r$ be the maximal radius satisfying the following: $r\le 2|z_0|$ and for every $z\in B(z_0,r)$, if $s$ denotes the $\mu_0$-binding time of the pair $((\lambda z_0,\xi_1(z_0,0)), (\lambda z,\xi_1(z,0))$, then $s\geq l-1$. By (\ref{**}), for each $1\le j<l$, 
$$2|\xi_j(z,0)|\ge |\xi_j(z_0,0)|\ge |z_0|^{k/d} e^{-\alpha j}.$$
Let $\varepsilon>0$ be a small constant such that $2(1-\varepsilon)^{k/d}>1$. Let $r_1=\min (r, \varepsilon |z_0|).$ 
Then for $z\in B(z_0, r_1)$, we have 
$$4|\xi_j(z,0)|\ge |z|^{k/d} e^{-\alpha j}$$ for all $1\le j<l$. 
Thus $B(z_0, r_1)\subset \widehat{K}_{\alpha, m, l}$. 
By Lemma \ref{relation}, for each $z\in B(z_0, r_1)$, 
$$\left|\frac{\varphi_l'(z)}{kz^{k-1}X_0 D_{l-1}(z)}-1\right|<\frac{1}{2}, \;\forall z\in B(z_0,r),$$
where $D_i(z)=Df^i(f(z,0))(v)$.
Then $\varphi_l$ is univalent in $B(z_0, r_1)$ and it suffices to show that $|\varphi'(z_0)| r_1\ge C |z_0|^{k/d} e^{-\alpha \omega l}$ provided that $m,l$ are large enough. By the Koebe distortion theorem, we only need to show that 
\begin{equation}\label{4.7-2}
|z_0^{k-1} r_1 D_{l-1} (z_0)|\ge C |z_0|^{k/d} e^{-\alpha \omega l}.
\end{equation}

{\bf Case 1.} $r>|z_0|$. Then $r_1=\varepsilon |z_0|$. In this case $0\in B(z_0,r)$, so by the definition of $r$,
\begin{equation}\label{4.7-1}
|Df_0^{l-1} (c(0))|\asymp |D_{l-1} (z_0)|.
\end{equation}

{\em Subcase 1.1.} Assume that $|f_0^l(0)|\ge 2|z_0|^{k/d} e^{-\alpha l/2}.$ Then since $|\xi_l(z_0,0)|\leq |z_0|^{k/d} e^{-\alpha l}$, we have  $$|\xi_l(0,0)-\xi_l(z_0,0)|\ge |z_0|^{k/d} e^{-\alpha l/2}.$$
Then by (\ref{4.7-1}) and Proposition \ref{prop:levin} (1),  we have 
$$W:=2|\xi_1(z_0,0)-\xi_1(0,0)|+\sum_{i=1}^{l-1} \frac{2|c(\lambda^{i-1}z_0)-c(0)|}{|D_i(z_0)|}\le C_0|z_0|^k,$$
where $C_0>0$ is a constant.
Then by Lemma \ref{**} we have 
$$D_{l-1}(z_0)\geq \frac{|\xi_l(0,0)-\xi_l(z_0,0)|}{W}\geq \frac{|z_0|^{k/d} e^{-\alpha l/2}}{C_0|z_0|^k}.$$
Hence 
$$|z_0^{k-1}r_1D_{l-1}(z_0)| \ge   |z_0|^{k/d} e^{-\alpha l/2} \varepsilon/C_0,$$ hence (\ref{4.7-2}) holds.

{\em Subcase 1.2.} Assume that $|f_0^l(0)|<2|z_0|^{k/d} e^{-\alpha l/2}$.
Then by Proposition \ref{prop:levin} (1),
$$	e^{-\alpha l/2} \preceq |Df_0^{l}(c(0))|= |Df_0^{l-1}(c(0))| d|f_0^l(0)|^{d-1}$$
and 
$$|Df_0^{l-1}(c(0))| d|f_0^l(0)|^{d-1}\leq |Df_0^{l-1}(c(0))| 2^d d |z_0|^{k(d-1)/d} e^{-\alpha l (d-1)/2},$$
which, by (\ref{4.7-1}),  again implies that $$|z_0^{k-1}r_1D_{l-1}(z_0)|  \succeq  |z_0|^{k/d} e^{-\alpha l/2},$$ hence (\ref{4.7-2}) holds.

{\bf Case 2.} $r\le |z_0|$. Then $r\asymp r_1$.
By the maximality of $r$ and the definition of the binding time, there is a minimal integer $2\leq n\leq l-1$ such that for $z\in \overline{B(z_0,r)}$,  
\begin{equation*}
|\xi_n(z_0,0)-\xi_n(z,0)|= \frac{\mu_0\min(|\xi_n(z,0)|, |\xi_{n}(z_0,0)|)}{ n^2}\ge \frac{\mu_0|\xi_n(z_0,0)|}{2 n^2}.
\end{equation*}
It follows that
$$|D_{n-1}(z_0)|r_1\asymp |D_{n-1}(z_0)|r\succeq n^{-2} |\xi_n(z_0,0)|\ge l^{-2}|\xi_n(z_0,0)|.$$

Note that $\{f^j(z_0,0)\}_{j=n}^l$ is tame, and $|\xi_l(z_0,0)|< |\xi_j(z_0,0)|$ for all $n\le j<l$.  By the last statement of Theorem~\ref{thm:derivatives},
$$|D_{l-1}(z_0)|/|D_{n-1}(z_0)|\succeq e^{-\alpha (l-n)/2},$$ hence
\begin{equation}\label{eqn:Dl-1z0r} 
D_{l-1}(z_0)r_1 \succeq l^{-2} |\xi_n(z_0,0)| e^{-\alpha (l-n)/2}.
\end{equation}

{\em Case 2.1.} $|\xi_n(z_0,0)|\ge |z_0|^{k/d} e^{-\alpha l/4}.$ 
Then 
$$D_{l-1}(z_0) r\succeq l^{-2} |z_0|^{k/d} e^{-3\alpha l/4}\succeq |z_0|^{k/d} e^{-4\alpha l/5},$$
provided that $l$ is large enough. 

{\em Case 2.2.} $|\xi_n(z_0,0)|< |z_0|^{k/d} e^{-\alpha l/4}$. Then by Lemma~\ref{p}, $l-n\asymp l$. Since $z_0\in K_{\alpha,m,l}$, we have 
$|\xi_n(z_0,0)|\ge |z_0|^{k/d} e^{-\alpha n}.$
By (\ref{eqn:Dl-1z0r}),
$$|D_{l-1} (z_0) r_1|\succeq l^{-1} |z_0|^{k/d} e^{-\alpha (l+n)/2} \succeq e^{-\alpha \omega l},$$
for a suitably chosen $\omega$, provided that $l$ is large enough. 
%
%
%
%
\end{proof}

For each integer $m_1\ge m$, let
$$K_{\alpha,m,l}^{m_1}=\{z\in K_{\alpha,m,l}: |\lambda|^{m_1+1} r_0\le |z|< |\lambda|^{m_1} r_0\}.$$
\par The following estimate of the volume of $K^{m_1}_{\alpha,m,l}$ is crucial in the proof of Theorem  \ref{slow}.
\begin{lemma}\label{K}
Let $f(z,w)=(\lambda z, w^d+c(z))$ as in (\ref{uni}). Assume $(0,0)$ is not a periodic point. Then for  every $\alpha>0$ satisfying $e^{-d\alpha}>|\lambda|^k $, there exists $\gamma=\gamma(\alpha)>0$ such that  for  $m$ and $l$ large, the following holds for all $m_1\ge m$:
\begin{equation*}
\frac{\text{Vol}\;(K_{\alpha,m,l}^{m_1})}{\text{Vol}\;(B(0,|\lambda|^{m_1}r_0))}\leq e^{-\gamma l}.
\end{equation*}
\end{lemma}
\begin{proof}
Fix $m_1\ge m$. Let $V_j'$, $j=1,2,\ldots$, be the connected components of the set $\varphi_l^{-1}(B(0, |\lambda|^{km_1/d} e^{-\alpha\omega l})$ such that $V_j'\subset B(0,2|\lambda|^{m_1}r_0)$ and such that $$\varphi_l: V_j'\to B(0, |\lambda|^{km_1/d} e^{-\alpha\omega l})$$ is a conformal map. Let $$V_j=\{z\in V_j': |\varphi_l(z)|\le |\lambda|^{km_1/d} e^{-\alpha l}\}.$$
By the Koebe distortion theorem,
$$\frac{\text{area}(V_j)}{\text{area}(V'_j)}\le C e^{-2\alpha (1-\omega) l}.$$
The previous lemma implies that $K_{\alpha,m,l}^{m_1}\subset \bigcup_j V_j$.
Since $V_j'$ are pairwise disjoint, we obtain the desired estimate.
\end{proof}
%
%
\medskip
\par Now we can prove Theorem  \ref{slow}.
\medskip
\par {\em Proof of Theorem  \ref{slow}:} It is sufficient to show that for every $\alpha>0$ small, and for every $\varepsilon>0$, we have $\text{Vol} \;(B(0,r_0)\setminus \Lambda_{2\alpha})<\varepsilon$. 
We shall use Lemma \ref{set} to estimate $\text{Vol} \;(B(0,r_0)\setminus \Lambda_{2\alpha})$. By Lemma \ref{p}, there exits a constant $\theta=\theta(\alpha)>0$ such that for every $m\geq 0$ and every $z\in B(0,|\lambda|^mr_0)$, $|\xi_l(z)|> e^{-\alpha l}|\lambda|^{km/d}$ provided $l\leq \theta m$. In other word, $K_{\alpha,m,l}=\emptyset$ provided $l\leq \theta m$. By Lemma~\ref{K},
$$\text{area}(K_{\alpha,m,l})\le \sum_{m_1=m}^\infty e^{-\alpha \gamma l} \text{area} (B(0, |\lambda|^{m_1} r_0))\le Ce^{-\alpha \gamma l} \text{area} (B(0, |\lambda|^m r_0)).$$
Thus by Lemma \ref{set} (2), for $m$ large enough we have
\begin{equation*}
\text{Vol}\;(B(0,|\lambda|^mr_0)\setminus \Omega_{\alpha,m})\leq \sum_{l=\theta m}^{+\infty} \text{Vol}\;(K_{\alpha,m,l})\leq Qe^{-\gamma \theta m} \text{Vol}\;(B(0,|\lambda|^mr_0),
\end{equation*}
where $Q:=C \sum_{l=1}^{+\infty} e^{-\gamma  l}$.
\medskip
\par Thus by Lemma \ref{set} (1), for $N$ large enough we have
\begin{align*}
\text{Vol}\;(B(0,r_0)\setminus \Lambda_{2\alpha})&\leq \sum_{m=N}^{+\infty} |\lambda|^{-2m}\text{Vol}\;(B(0,|\lambda|^mr_0)\setminus \Omega_{\alpha,m})\\
&\leq \sum_{m=N}^{+\infty} |\lambda|^{-2m}Qe^{-\gamma \theta m} \text{Vol}\;(B(0,|\lambda|^mr_0)\\
&= \sum_{m=N}^{+\infty} Q\pi|r_0|^2e^{-\gamma \theta m}  \leq \varepsilon.
\end{align*}
\medskip
\par The conclusion follows.
\qed
\bigskip
\section{Non-wandering Fatou components}
In this section we prove Theorem \ref{main}. The strategy is the following. We will first show that if a vertical disk $D\subset L_{z_0}$ centered at $x_0=(z_0,w_0)$ is very close to the invariant line $L$ (which means that the radius of $D$ is much larger than $|z_0|$), then for every $n\geq 0$, $f_0^n(w_0)\in \xi_n(D)$. The proof of this proposition is again by using the binding argument. Now assume by contradiction there is a wandering Fatou component, by Theorem \ref{thm:derivatives} and Theorem \ref{slowapp} we can select a vertical disk $D$ contained in a wandering Fatou component,  which is very close to the invariant line $L$. Finally   we will get a contradiction.
\medskip
\par More precisely, we first prove the following result.
\begin{proposition}\label{radius}
Let $f(z,w)=(\lambda z, w^d+c(z))$ be as in (\ref{uni}) such that $f_0$ has no attracting nor superattracting cycle in $\mathbb{C}$. There exists $\rho>0$ such that the following hold. Let $D\subset L_{z_0}$  be a vertical disk centered at $x=(z_0,w_0)$ of radius $\delta\in (0, \rho)$ with $|z_0|<\delta^{2d}$.  Then for every $0<\lambda_0<1$, there is a constant $C=C(\lambda_0)>0$  such that for each $n\geq 1$,
\begin{equation*}
B(f_0^n(w_0), C\lambda_0^n\delta)\subset \xi_n(D).
\end{equation*}
\end{proposition}

\par The proof of this proposition will be given after we prove two lemmas.
\begin{lemma}\label{lemma4} For any $\delta_0>0$ and $\lambda_0\in (0,1)$,
there exist positive integers  $N=N(\delta_0,\lambda_0)$, $\rho=\rho(\delta_0,\lambda_0)>0$ and $C=C(\delta_0,\lambda_0)>0$  such that for every $w_0\in \mathbb{C}$ such that the $f_0-$orbits of $w_0$ is bounded, the following hold. 
\begin{enumerate}
\item If $|f_0^i(w_0)|>\delta_0/2$ for every $0\leq i< N$ and $|z_0|^k\le \delta^2$ for some $\delta\in (0, \rho)$, then
\begin{equation*}
B(f_0^N(w_0), \lambda_0^N\delta)\subset \xi_N(\{z_0\}\times B(w_0, \delta)),
\end{equation*}
and
\begin{equation*}
B(f_0^j(w_0), C\lambda_0^j\delta)\subset \xi_j(\{z_0\}\times B(w_0, \delta)), \,\, 1\le j<N.
\end{equation*}
\item If $n\le N$ is a positive integer such that $|f_0^i(w_0)|>\delta_0/2$ for all $0\le i<n$ and $|f_0^n(w_0)|\le \delta_0/2$, and $|z_0|^k\le \delta^2$ for some $\delta\in (0, \rho)$, then
\begin{equation*}
B(f^n(w_0), \kappa_0\lambda_0^n \delta/4)\subset \xi_n(\{z_0\}\times B(w_0,\delta)),
\end{equation*}
where $\kappa_0=\kappa_0(\lambda_0)>0$ is a constant, and
\begin{equation*}
B(f_0^j(w_0), C\lambda_0^j\delta)\subset \xi_j(\{z_0\}\times B(w_0, \delta)), \,\, 1\le j<n.
\end{equation*}
\end{enumerate}
\end{lemma}
\begin{proof}
Without loss of generality, we may assume $\lambda_0\in (|\lambda|,1)$. Let $\kappa=\kappa(\lambda_0^{1/2},\delta_0/2)$ and $\kappa_0=\kappa_0(\lambda_0^{1/2})$ be given by Proposition~\ref{prop:1-dim-exp}. Choose $N$ such that if $\kappa\lambda_0^{N/2}>4 \lambda_0^{N}$.
Let $D=\{z_0\}\times B(w_0, \delta)$.
For every $x\in D$, let $s_x$ be the $\mu_0$-binding time of the pair $(x,(0, w_0))$. Let $s=\min_{x\in D} s_x$. Provided that $\delta$ is small enough, we have $s>N$.
%

Lemma \ref{*}  implies that $w\mapsto \xi_m(z_0, w)$ is univalent in $B(w_0, \delta)$ and hence $\xi_m(D)$ contains a disk centered at $\xi_m(x_0)$ with radius at least $R_m:=1/2 |(f_0^m)'(w_0)|\delta$.  Let $s_1=N$ in Case (1) and $s_1=n$ in Case (2). Let $x_0=(z_0,w_0)$ and $y_0=(0,w_0)$. By (\ref{eqn:W00}) and Proposition \ref{prop:1-dim-exp} we have
$$W(x_0,y_0, s_1)\le \sum_{i=1}^{s_1}  \frac{4|\lambda^{k(i-1)}z_0^k|}{|(f_0^i)'(w_0)|}\le C |z_0|^k,$$
where $C$ is a constant depending on $\lambda_0$ and $\delta_0$.
By Lemma \ref{**} ,
\begin{align*}
|\xi_{s_1}(x_0)-f_0^{s_1}(w_0)|&\leq |(f_0^{s_1})'(w_0)|W(x_0,y_0,s_1)\le C|(f_0^{s_1})'(w_0)| |z_0|^k< \frac{R_{s_1}}{2},
\end{align*}
provided that $\delta$ is small enough. 
Therefore, $$\xi_{s_1}(D) \supset B(f_0^{s_1}(w_0), R_{s_1}/2).$$
The lemma follows since in Case (1) we have $|Df_0^{s_1}(w_0)|> \lambda_0^N$ and in Case (2) we have $|Df_0^{s_1}(w_0)|\ge \kappa_0 \lambda_0^n$ and in both cases, $|Df_0^j(w_0)|\ge C\lambda_0^j$ for $1\le j<s_1$.
\end{proof}
\begin{lemma}\label{lemma5}
For each $\lambda_0\in (0,1)$ and each $K>0$, there exist constants $C=C(\lambda_0, K)>0$ and $\delta_0>0$ such that the following holds. Assume $|w_0|\leq \delta_0$ and $|z_0|^k< \delta^{d+1}$ for some $0<\delta<\delta_0$.  Then there exists a positive integer $n$ such that
\begin{equation*}
B(f_0^{n+1}(w_0), K\lambda_0^{n+1}\delta)\subset \xi_{n+1}(\{z_0\}\times B(w_0, \delta)),
\end{equation*}
and
\begin{equation*}
B(f_0^j(w_0), C\lambda_0^{j}\delta^d)\subset \xi_j(\{z_0\}\times B(w_0, \delta)), \,\, 1\le j\le n.
\end{equation*}
\end{lemma}

\begin{proof} 
Fix constants $\lambda_0\in (0,1)$, $K>0$ and a large positive integer $s_0$ such that $\lambda_0^{-s_0/2}\ge C_3^{-1}K$, where $C_3$ is a constant to be determined below. 
We may and will assume $\lambda_0>|\lambda|$.

For each $j\ge 0$, write $x_j=(z_j, w_j)=f^j(z_0,w_0)$ and 
$$D_j=f^j(\{z_0\}\times B(w_0,\delta))\subset \{z_j\}\times \mathbb{C}.$$
Then $D_1=\{z_1\}\times \xi_1(D_0) \supset \{z_1\}\times B(w_1, \delta_1)$ where
\begin{equation}\label{eqn:delta1}
\delta_1=\theta \delta \max(|w_0|,\delta)^{d-1}.
\end{equation}
for some uniform constant $\theta>0$. 
Let $\mu_1$ be as in (\ref{eqn:mu1}) and let $$s=\inf\{b_{\mu_1}((z,w), (0, c(0))): |z|\le |\lambda||z_0|, w\in \xi_1(D_0)\}.$$

{\bf Claim 1.} Provided that $\delta_0$ is small enough, for each $1\le j\le s$, 
$$\xi_{j+1}(D_0)\supset B(f_0^{j+1}(w_0), R_{j+1}/4),$$
where $$R_{j+1}=|Df^j(z_1,w_1)(v)|\delta_1.$$

Indeed, as in the proof of the previous lemma, for each integer $0\le j\le s$,
$\xi_{j+1}(D_0)\supset \xi_j(D_1)\supset B(\xi_{j+1}(x_0), R_{j+1}/2)$. So it suffices to prove 
$$|\xi_{j+1}(x_0)-f_0^{j+1}(w_0)|\le R_{j+1}/4.$$
To prove this, we first apply Proposition~\ref{prop:levin} (1) to obtain a constant $C_1>0$ such that 
\begin{equation}\label{eqn:df0c0}
|Df_0^i(c(0))|\ge 2C_1\lambda_0^i\mbox{ for all }i\ge 0.
\end{equation}  
By (\ref{eqn:W00}), 
$$W((z_1,w_1), (0, w_1),j)\le  4\sum_{i=1}^j \frac{|\lambda|^{(i-1)k}|z_1|^k}{|Df_0^i(c_0)|}\ll \delta_1.$$
By (\ref{eqn:mu0mu1}), 
$$b_{\mu_0}((z_1, w_1), (0, w_1)\ge \min (b_{\mu_1}((z_1,w_1), (0, c(0))), b_{\mu_1} ((0,w_1), (0, c(0))))\ge s.$$ 
By Lemma~\ref{**}, it follows that 
$$|\xi_j(z_1, w_1)-f_0^{j+1}(w_0)|\le |Df^{j}(z_1,w_1)(v)| W((z_1, w_1), (0, w_1),j)\le R_{j+1}/4.$$
The claim is proved. 
\medskip

By Lemma~\ref{*}, for each $1\le j\le s$, $|Df^j(z_1,w_1)|\ge |Df_0^j(c(0))|/2$. So by (\ref{eqn:df0c0}) and (\ref{eqn:delta1}), we obtain 
\begin{equation}\label{eqn:xij+1}
\xi_{j+1}(D_0)\supset B(f_0^{j+1}(w_0), C_1\theta \lambda_0^i \delta^d).
\end{equation}

{\bf Claim 2.} Provided that $\delta_0>0$ is small enough, there exists a positive integer $n\le s$ such that  $R_{n+1}> 2K\lambda_0^{n}\delta.$

To prove the claim, 
let $(z_1', w_1')$ be such that $|z_1'|\le |\lambda z_0|$ and $w_1'\in \xi_1(D_0)$ with $$b_{\mu_1} ((z_1', w_1'), (0, c(0))=s.$$
By Lemma~\ref{lem:bdcv}, there exists a positive integer $s_0\le n\le s$ such that
$$|Df^n(z'_1,w'_1)(v)|\ge \lambda_0^{n/2} \delta'^{-(d-1)},$$
where
$$\delta'^d=\min (|w'_1-c(0)|, |z'_1|^k).$$
Let $w_0'\in B(w_0,\delta)$ be such that $w'_1=\xi_1(z_0, w_0')$. Then
\begin{equation*}
|w'_1-c(0)|=|w'^d_0+ c(z_0)-c(0)| \le (|w_0|+\delta)^d +2 |z_0|^k\le (|w_0|+\delta)^d+ 2\delta^{d+1}.
\end{equation*}
Since $|z_1|^k\le |z_0|^k<\delta^{d+1},$
there exists a constant $C_2>0$ such that $\delta'\le C_2\max(|w_0|,\delta).$
%
It follows that, by the definition of $\delta_1$,
$$|Df^n(z'_1,w'_1)(v)|\delta_1\ge \lambda_0^{n/2} C_2^{-(d-1)} \theta \delta.$$
By Lemma~\ref{*},
$$|Df^n(z_1',w_1')(v)|\le 2|Df_0^n(c(0))|\le 4 |Df^n(z_1,w_1)(v).$$
Therefore,
$$R_{n+1}=|Df^n(z_1,w_1)(v)|\delta_1\ge \frac{1}{4} |Df^n(z_1', w_1') (v)| \delta_1\ge C_3 \lambda_0^{n/2}\delta,$$
where $C_3$ is a constant. By our choice of $s_0$ (at the beginning of the proof), the implies the claim.

By Claims 1 and 2 and (\ref{eqn:xij+1}), the lemma follows. 
\end{proof}

\medskip
\par Now we can prove Proposition \ref{radius}.
\medskip
\begin{proof}[Proof of Proposition \ref{radius}]  We may assume that $\lambda_0^{d+1} > |\lambda|$. Fix such $\lambda_0\in (|\lambda|,1)$ and let $\kappa_0=\kappa_0(\lambda_0)>0$ be given by Lemma~\ref{lemma4} (2). Choose $K=4/\kappa_0$. Let $\delta_0=\delta_0(\lambda_0, K)>0$ be given by Lemma \ref{lemma5}. Let $\delta>0$ be small.
Suppose that we have found a nonnegative integer $m$ such that
$$\xi_m(D)\supset B(f_0^m(w_0), \lambda_0^m \delta)$$
and $$\xi_{j}(D)\supset B(f_0^j(w_0), C\lambda_0^j\delta^d ), \, 1\le j<m.$$
Note that $m=0$ satisfies these properties.
For $\delta(m)=\lambda_0^m\delta$, we have 
$$|z_m|=|\lambda|^m |z_0|\le |\lambda|^m \delta^{2d} < \delta(m)^2,$$
so that we may apply Lemma~\ref{lemma4} to the $f$-orbit of $(z_m, f_0^m(w_0))$. 
If $|f_0^j(w_0)|\ge \delta_0/2$ for all $j\ge m$, then by Lemma \ref{lemma4}  (1) we obtain $$\xi_j(D)\supset\xi_{j-m}(\{z_m\}\times \xi_m (D))\supset B(f_0^j(w_0), C\lambda_0^j\delta)$$ for all $j$ and hence we are done. Otherwise, let $t\ge m$ be minimal such that $|f_0^t(w_0)|<\delta_0/2$. Then by  Lemma \ref{lemma4}  (2)  we obtain $$\xi_t(D)  \supset B(f_0^t(w_0), \kappa_0 \lambda_0^t \delta/4)$$ together with
$$\xi_j(D)\supset B(f_0^j(w_0), C\lambda_0^j \delta) \mbox{ for } m\le j<t.$$ 
For $\delta(t)=\kappa_0 \lambda_0^t \delta$, we have 
$$|z_t|=|\lambda|^t |z_0|\le |\lambda|^t \delta^{2d} <\delta(t)^{d+1}$$
holds, so applying Lemma~\ref{lemma5} to the $f$-orbit of $(z_t, f_0^t(w_0))$, we obtain a positive integer $m'>t$ such that
$$\xi_{m'}(D) \supset B(f_0^{m'}(w_0), K\kappa_0\lambda_0^{m'} \delta/4)\supset B(f_0^m(w'),\lambda_0^{m'}),$$
and $$\xi_{j}(D)\supset B(f_0^j(w_0),  C\lambda_0^j\delta^d)\mbox{ for } t< j<m'.$$
Repeat the argument for $m'$ instead of $m$ and continue, the conclusion follows. 
\end{proof}

\medskip
\begin{remark}
Proposition \ref{radius} is even new in dimension one. Taking $D\subset L$, we get the following one-dimensional result: let $f_0(w)=w^d+c$ as in (\ref{uni-one}), then for every $0<\lambda_0<1$, there are constants $C=C(\lambda_0)>0$ and $\delta_0=\delta_0(\lambda_0)>0$ such that for every $w$ not be contained in an attracting basins of $f_0$ ,  $\delta<\delta_0$ and $n\geq 0$, we have
\begin{equation*}
B(f_0^n(w),C\lambda_0^n\delta^d)\subset f_0^n(B(w,\delta)).
\end{equation*}
\medskip
\par This improves a result of Denker-Przytycki-Urbanski \cite[Lemma 3.4]{denker1996transfer} for unicritical polynomials. In \cite{denker1996transfer}, it is proved that for rational map $f_0$, there are constants $0<L<1$, $\rho>0$ and $\delta_0>0$ such that  $B(f_0^n(w),L^n\delta^\rho)\subset f_0^n(B(w,\delta))$  for $w$ in the Julia set of $f_0$ and $\delta<\delta_0$.  But the constants $L$ and $\rho$ are not controlled.
\end{remark}
\medskip
\par We are now ready to prove Theorem \ref{main}.
\medskip
%
\par {\em Proof of Theorem \ref{main}}.
First it is proved in \cite{ji2020non} that every Fatou component of $f_0$ can be extended to a two-dimensional Fatou component of $f$. These kind of Fatou components are clearly non-wandering, since their restrictions on  $L$ is non-wandering, by Sullivan's theorem \cite{sullivan1985quasiconformal}. Next we show that every Fatou component of $f$ in a small neighborhood of $L$ is an  extension of a Fatou component of $f_0$.
\medskip
\par Since $f_0$ is unicritical, if $f_0$ has an attracting nor superattracting cycle in $\mathbb{C}, $ then $f_0$ is hyperbolic. In this case it is not hard to show, by the shadowing lemma,  that the Fatou set of $f$ is the union of basins of attracting cycles. See also \cite{peters2018fatou} or \cite{ji2019nonuni}. In the following we assume that $f_0$ has no attracting nor superattracting cycle in $\mathbb{C}$.
\medskip
\par We argue by contradiction.   Assume there is a Fatou component $\Omega$ such that $\Omega$ is not an extension of a one-dimensional Fatou component, i.e. $\Omega$ is wandering. Clearly, the $f$-orbit of $\Omega$ is uniformly bounded. Fix a small constant $\alpha>0$ such that 
\begin{equation}\label{eqn:alpha}
|\lambda|^k< e^{-2d^2\alpha}
\end{equation}
and fix $\lambda_0$ such that $|\lambda|^k e^{2d^2 \alpha} <\lambda_0<1$.
Since $\Omega$ is open, it has positive volume. By Theorem \ref{slowapp}, there exists $x'=(z_0', w_0')\in \Omega$ such that for $n$ large enough, $|\xi_n(x')|\geq e^{-n\alpha}$. 
Then there exists an integer $N>0$ such that the orbit of $x_N':=(z_N', w_N')$ is tame. (See the definition of tame orbits in section 1.)  By Theorem~\ref{thm:derivatives}, for every $n\geq 1$
\begin{equation}\label{5.14}
|Df^n(x_N')(v)|\geq C\lambda_0^n e^{-(d-1)(n+N)\alpha},
\end{equation}
where $C=C(\lambda_0)>0$. 

{\bf Claim.} There exists a constant $\kappa>0$ and an arbitrarily large $n$ such that $$\{z_n'\}\times B(w_n', \kappa \lambda_0^n e^{-(d+1)\alpha n})\subset \Omega_n=:f^n(\Omega).$$

To prove this claim, let $\varphi_m(w)=\xi_m(z_N', w)$ for each $m\ge 0$ and let $\varepsilon_0>0$ be so small that $\{z'_N\}\times B(w_N',\varepsilon_0)\subset \Omega_N$.  We distinguish two cases.

{\bf Case 1.} There exists $\varepsilon\in (0,\varepsilon_0)$ such that $\varphi_m$ is univalent on $B(w_N', \varepsilon e^{-\alpha m})$ for all $m\ge 1$. Then by Koebe $1/4$ Theorem, $\{z_{N+m}'\}\times B(w_{N+m}', R_m)\subset \Omega_{N+m},$ where $$R_m=|Df^m(z_N', w_N')(v)|e^{-\alpha m} \varepsilon/4.$$ Then the conclusion of the claim holds for all $n$ large enough.

{\bf Case 2.} For each $\varepsilon\in (0,\varepsilon_0)$ there exists a minimal $m=m(\varepsilon)$ such that $\varphi_{m+1}$ is not univalent on $B(w_N', \varepsilon e^{-\alpha (m+1)})$. Then $m\to\infty$ as $\varepsilon\to 0$. By the minimality of $m$, $\varphi_m$ is univalent on $B(w_N', \varepsilon e^{-\alpha m})$ and  we have $\left\{w_{m+N}',0\right\}\subset \varphi_m(B(w_N', \varepsilon e^{-\alpha (m+1)})) $, so that its diameter is at least $e^{-\alpha (m+N)}$.
By the Koebe distortion theorem,  it follows that the image $\varphi_m(B(w_N', \varepsilon e^{-\alpha (m+1)}))$ contains a ball centered at $w_{m+N}'$ and of radius at least $\rho(\alpha)e^{-\alpha (m+N)}$, where $\rho(\alpha)>0$ is a constant. Then $n=m+N$  satisfies the requirement of the claim.
The claim is proved.

Take $n_0$ large for which the property in the Claim is satisfied. Rename $z_{n_0}'$ as $z_0$ and choose $w_0\in F(f_0)$ very close to $w_{n_0}'$ such that $w_0\in B(w_{n_0}', \frac{1}{2}\kappa \lambda_0^{n_0} e^{-(d+1)\alpha n_0})$.
(We can choose such $w_0$ since the Fatou set $F(f_0)$ is dense in $\mathbb{C}$.) Let $\delta:= \frac{1}{2}\kappa \lambda_0^{n_0} e^{-(d+1)\alpha n_0}$ and let $D:=\{z_0\}\times B(w_0,\delta)$, which is a vertical disk.  By the Claim,  $D$ is {\bf contained} in the Fatou component $\Omega_{n_0}$.  Moreover by the choice of $\alpha$ in (\ref{eqn:alpha}), the vertical disk $D$ satisfies the condition in Proposition \ref{radius}, i.e. we have $\delta\in (0,\rho)$ and  $|z_0|<\delta^{2d}$, provided that $n_0$ was chosen large enough.

By Proposition~\ref{radius}, it follows that for every $n\geq 1$,
\begin{equation}\label{5.15}
 f_0^n(w_0)\in \xi_n(D)\;\;\text{and}\;\; f^n(D)\subset \Omega_{n_0+n}.
\end{equation}
By the classification of Fatou components in one dimension, $f_0^n(w_0)$ converges either to a parabolic cycle, a Siegel disk or to $\infty$.  In the latter two cases, by (\ref{5.15}), for $n$ large enough $f^n(D)$ must intersect with an extension of a one-dimensional Fatou component, a contradiction. In the parabolic cycle case, it is proved in \cite[Theorem 3.3]{ji2020non} , which is a special case of  Ueda \cite{ueda1986local} section 7.2, that there exist attracting petals  of the form $U_j=\left\{|z|<\varepsilon\right\}\times (U_j\cap \left\{z=0\right\})$, such that  $f^n(x)$ converges to the parabolic cycle $p$ implies $f^N(x)\in U_j$ for some $N$ and $j$. Moreover these attracting petals are contained in the extension of the one-dimensional parabolic basin. Thus by (\ref{5.15}), for $n$ large enough $f^n(D)$ must intersect with an attracting petal, a contradiction. The proof is completed.
\qed
\bigskip
\par {\em Proof of Theorem \ref{main2}}. By Lilov's result \cite{lilov2004fatou}, $f$ does not have wandering Fatou components in the attracting basin of the line at infinity.  Let $\Omega$ be a Fatou component with bounded orbit, we need to show that $\Omega$ is non-wandering. Let $\pi$ be the projection to the $z-$coordinate, then $\pi(\Omega)$ is contained in a bounded Fatou component of $p$.  Since $\pi(\Omega)$ is connected, $\pi(\Omega)$ is contained in a Fatou component  of $p$. Since $p$ is a polynomial without parabolic periodic points and Siegel periodic points, all Fatou components of $p$ are attracting or superattracting basins. The superattracting case is again covered by \cite{lilov2004fatou}. We can assume that $\pi(\Omega)$ is contained in the basin of a attracting periodic point $z_0$ of $p$ with periods $s$. Replacing $f$ by $f^s$, $z_0$ will be a fixed point of $p$. By our assumption that the critical curve of $f^s$ in $\mathbb{C}^2$ has a unique  transversal intersection with the vertical  line $L:=\left\{z=z_0\right\}\subset \mathbb{C}^2$, $f^s$ can be conjugated in a neighborhood $U$ of the vertical line $\left\{z=z_0\right\}$, to the form as in (\ref{uni}). Clearly there exists $N\geq 1$ such that $U\cap f^N(\Omega)\neq \emptyset$.  Apply Theorem \ref{main} we get that $U\cap f^N(\Omega)$ is contained  in a non-wandering Fatou component, This implies that $\Omega$ itself is non-wandering.

\qed

\bigskip
\par {\em Proof of Theorem \ref{main3}}. It is clear that $f$ can be extended to a holomorphic endomorphism on $\mathbb{P}^2$. We need to show that when $|\lambda|<1$, $f$ does not have wandering Fatou components. When $|\lambda|<1$, $p$ has no parabolic periodic points, nor Siegel periodic points. The only attracting periodic point of $p$ is $0$, which is fixed. The critical curve of $f$ in $\mathbb{C}^2$ has two components $\left\{z=-\lambda/2\right\}$ and $\left\{w=0\right\}$. So the critical curve has a unique transversal intersection with the invariant vertical line $\left\{z=0\right\}$. (The intersection point is $(0,0)$). So our map $f$ satisfies all the assumptions in Theorem \ref{main2}, and we conclude the result  by applying Theorem \ref{main2}.

\qed
\bigskip
\section{The multicritical case}
In this section we discuss the situation when $f$ is multicritical (i.e. when $f$ is a regular  polynomial skew product with an invariant attracting line). For a  regular  polynomial skew product with an invariant attracting line,  we may locally conjugate $f$ to the following form
\begin{equation}\label{general}
f(z,w)=(\lambda z, F(z,w)),
\end{equation}
where $|\lambda|<1$ and $F(z,w)=w^d+\sum_{i=0}^{d-1}c_i(z) w^i $ is a polynomial in $w$ with coefficients holomorphic in $z$, $d\geq 2$. Let $f_0$ be the one-dimensional map $f_0(w)=w^d+\sum_{i=0}^{d-1}c_i(0)w^i,$ which is the restriction of $f$ on the invariant line. It is natural to ask whether our Theorem \ref{main} hold in this more general  setting.
\begin{problem}\label{problem}
Let $f(z,w)=(\lambda z, F(z,w))$ as in (\ref{general}). Is that true that every Fatou component of $f$ is an extension of a Fatou component of $f_0$?
\end{problem}

\medskip
\par A positive answer will imply that there is no wandering Fatou component. Notice that most of our techniques have  multicritical version (e.g. the binding argument and the parameter exclusion technique).  However there is a major difference between unicritical and multicritical polynomials: it was proved in \cite{levin2016lyapunov} that for unicritical polynomial $f_0$ the lower Lyapunov exponent of the critical value $c$ is always non-negative, provided $c$ is not contained in an attracting basin of $f_0$. But in the multicritical case, there are semi-hyperbolic polynomials carrying a critical value on the Julia set with $-\infty$ Lyapunov exponent, see Przytycki-Rohde \cite{przytycki1999rigidity}.

\medskip
\par However assuming the following two conditions, we expect our methods can apply  and the answer of Problem \ref{problem} is yes. Let  $$C(f):=\left\{x\in\mathbb{C}^2:Df(x) \;\text{is not invertible}\right\}$$ be the  critical curve,  and let $L:=\left\{z=0\right\}$.
\medskip
\par (1) {\em Lyapunov exponent}: for every critical value $c$ of $f_0$ such that  $c$ is contained in the Julia set of $f_0$, its lower Lyapunov exponent  satisfies
\begin{equation*}
 \chi_-(c):=\liminf_{n\to+\infty} \frac{1}{n}\log |Df_0^n(c)|>\log \lambda.
\end{equation*}
\par (2) {\em Non-degeneracy condition}: the critical curve $C(f)$  intersects $L$ transversally, and the following power series converges to a non-zero constant for every critical value $c$ of $f_0$ which is contained in the Julia set of $f_0$,
\begin{equation}\label{non-deg}
G(c)+\sum_{i=1}^{+\infty}\frac{\lambda^i G(f_0^i(c))}{(f_0^i)'(c)}\neq 0,
\end{equation}
\medskip
\par
where $G$ is a polynomial defined by $G(w):=\sum_{i=0}^{d-1} c_i'(0) w^i$.
\medskip
\par  We give some explanation of the second condition. By Proposition \ref{prop:levin}, for a unicritical polynomial $f_0(w)=w^d+c$ satisfying $f_0$ has no attracting cycle in $\mathbb{C}$, we have
\begin{equation*}
 1+\sum_{i=1}^{+\infty} \frac{\lambda^i} {(f_0^i)'(c)}\neq 0
\end{equation*}
for every $|\lambda|<1$. Thus if $f$ is unicritical of the form (\ref{uni}) such that $c'(0)\neq 0$, the polynomial $G$ in (\ref{non-deg}) is a non-zero constant, and the non-degeneracy condition (\ref{non-deg}) is automatically true. Similar non-degeneracy condition raises in other context, see for instance Tsujii \cite{tsujii1993positive} and Gao-Shen \cite{gao2014summability}.


\medskip
\begin{thebibliography}{10}
	
	\bibitem{astorg2023hyperbolicity}
	Matthieu Astorg and Fabrizio Bianchi.
	\newblock Hyperbolicity and bifurcations in holomorphic families of polynomial
	skew products.
	\newblock {\em American Journal of Mathematics}, 145(3):861--898, 2023.
	
	\bibitem{astorg2019wandering}
	Matthieu Astorg, Luka Boc~Thaler, and Han Peters.
	\newblock Wandering domains arising from {L}avaurs maps with {S}iegel disks.
	\newblock {\em Analysis and PDE}, 16(1):35--88, 2023.
	
	\bibitem{astorg2016two}
	Matthieu Astorg, Xavier Buff, Romain Dujardin, Han Peters, and Jasmin Raissy.
	\newblock A two-dimensional polynomial mapping with a wandering fatou
	component.
	\newblock {\em Annals of Mathematics}, 184(1):263--313, 2016.
	
	\bibitem{astorg2022dynamics}
	Matthieu Astorg and Luka~Boc Thaler.
	\newblock Dynamics of skew-products tangent to the identity.
	\newblock {\em Journal of the European Mathematical Society}, 2024.
	
	\bibitem{benedicks1985iterations}
	Michael Benedicks and Lennart Carleson.
	\newblock On iterations of {$1-ax^2$} on {$(-1, 1)$}.
	\newblock {\em Annals of mathematics}, 122(1):1--25, 1985.
	
	\bibitem{berger2022emergence}
	Pierre Berger and Sebastien Biebler.
	\newblock Emergence of wandering stable components.
	\newblock {\em Journal of the American Mathematical Society}, 36(2):397--482,
	2023.
	
	\bibitem{denker1996transfer}
	Manfred Denker, Feliks Przytycki, and Mariusz Urba{\'n}ski.
	\newblock On the transfer operator for rational functions on the {R}iemann
	sphere.
	\newblock {\em Ergodic Theory and Dynamical Systems}, 16(2):255--266, 1996.
	
	\bibitem{dujardin2021geometric}
	Romain Dujardin.
	\newblock Geometric methods in holomorphic dynamics.
	\newblock In {\em Proc. Int. Cong. Math}, volume~5, pages 3460--3482, 2022.
	
	\bibitem{gao2014summability}
	Bing Gao and Weixiao Shen.
	\newblock Summability implies {C}ollet-{E}ckmann almost surely.
	\newblock {\em Ergodic Theory and Dynamical Systems}, 34(4):1184--1209, 2014.
	
	\bibitem{hahn2021polynomial}
	David Hahn and Han Peters.
	\newblock A polynomial automorphism with a wandering {F}atou component.
	\newblock {\em Advances in Mathematics}, 382:107650, 2021.
	
	\bibitem{jakobson1981absolutely}
	Michael~V Jakobson.
	\newblock Absolutely continuous invariant measures for one-parameter families
	of one-dimensional maps.
	\newblock {\em Communications in Mathematical Physics}, 81(1):39--88, 1981.
	
	\bibitem{ji2020non}
	Zhuchao Ji.
	\newblock Non-wandering {F}atou components for strongly attracting polynomial
	skew products.
	\newblock {\em The Journal of Geometric Analysis}, 30(1):124--152, 2020.
	
	\bibitem{ji2019nonuni}
	Zhuchao Ji.
	\newblock Non-uniform hyperbolicity in polynomial skew products.
	\newblock {\em International Mathematics Research Notices},
	2023(10):8755--8799, 2023.
	
	\bibitem{levin2002analytic}
	Genadi Levin.
	\newblock On an analytic approach to the {F}atou conjecture.
	\newblock {\em Fundamenta Mathematicae}, 171:177--196, 2002.
	
	\bibitem{levin2016lyapunov}
	Genadi Levin, Feliks Przytycki, and Weixiao Shen.
	\newblock The {L}yapunov exponent of holomorphic maps.
	\newblock {\em Inventiones mathematicae}, 205(2):363--382, 2016.
	
	\bibitem{lilov2004fatou}
	Krastio Lilov.
	\newblock {\em Fatou theory in two dimensions}.
	\newblock PhD thesis, University of Michigan, 2004.
	
	\bibitem{peters2019fatou}
	Han Peters and Jasmin Raissy.
	\newblock Fatou components of elliptic polynomial skew products.
	\newblock {\em Ergodic Theory and Dynamical Systems}, 39(8):2235--2247, 2019.
	
	\bibitem{peters2018fatou}
	Han Peters and Iris~Marjan Smit.
	\newblock Fatou components of attracting skew-products.
	\newblock {\em The Journal of Geometric Analysis}, 28(1):84--110, 2018.
	
	\bibitem{peters2016polynomial}
	Han Peters and Liz~Raquel Vivas.
	\newblock Polynomial skew-products with wandering {F}atou-disks.
	\newblock {\em Mathematische Zeitschrift}, 283(1-2):349--366, 2016.
	
	\bibitem{przytycki1993lyapunov}
	Feliks Przytycki.
	\newblock Lyapunov characteristic exponents are nonnegative.
	\newblock {\em Proceedings of the American Mathematical Society},
	119(1):309--317, 1993.
	
	\bibitem{przytycki1998porosity}
	Feliks Przytycki and Steffen Rohde.
	\newblock Porosity of {C}ollet--{E}ckmann {J}ulia sets.
	\newblock {\em Fundamenta Mathematicae}, 155(2):189--199, 1998.
	
	\bibitem{przytycki1999rigidity}
	Feliks Przytycki and Steffen Rohde.
	\newblock Rigidity of holomorphic {C}ollet-{E}ckmann repellers.
	\newblock {\em Arkiv f{\"o}r Matematik}, 37(2):357--371, 1999.
	
	\bibitem{MR2769221}
	Roland K.~W. Roeder.
	\newblock A dichotomy for {F}atou components of polynomial skew products.
	\newblock {\em Conform. Geom. Dyn.}, 15:7--19, 2011.
	
	\bibitem{sullivan1985quasiconformal}
	Dennis Sullivan.
	\newblock Quasiconformal homeomorphisms and dynamics {I}. {S}olution of the
	{F}atou-{J}ulia problem on wandering domains.
	\newblock {\em Annals of mathematics}, 122(2):401--418, 1985.
	
	\bibitem{tsujii1993positive}
	Masato Tsujii.
	\newblock Positive {L}yapunov exponents in families of one dimensional
	dynamical systems.
	\newblock {\em Inventiones mathematicae}, 111(1):113--137, 1993.
	
	\bibitem{ueda1986local}
	Tetsuo Ueda.
	\newblock Local structure of analytic transformations of two complex variables,
	{I}.
	\newblock {\em Journal of Mathematics of Kyoto University}, 26(2):233--261,
	1986.
	
\end{thebibliography}
\end{document}